\documentclass[11pt]{amsart}
\usepackage{amsmath,amscd,amssymb,amsfonts,amsthm,stmaryrd,color,mathtools}
\usepackage{extarrows}
\usepackage[cmtip,matrix,arrow,color]{xy}

\newcommand{\cbrack}[2]{\llbracket #1, #2 \rrbracket}

\newtheorem{theorem}{Theorem}[section]
\newtheorem{corollary}[theorem]{Corollary}
\newtheorem{proposition}[theorem]{Proposition}

\theoremstyle{definition}
\newtheorem{definition}[theorem]{Definition}
\theoremstyle{remark}

\newtheorem{remark}[theorem]{Remark}

\newtheorem{example}[theorem]{Example}
\newtheorem{examples}[theorem]{Examples}

\newcommand{\cE}{\mathcal{E}}

\newcommand{\cI}{\mathcal{I}}

\newcommand{\cL}{\mathcal{L}}
\newcommand{\cM}{\mathcal{M}}
\newcommand{\cN}{\mathcal{N}}

\newcommand{\fg}{\mathfrak{g}}
\newcommand{\fX}{\mathfrak{X}}

\newcommand{\bA}{\mathbb{A}}
\newcommand{\bN}{\mathbb{N}}

\newcommand{\bZ}{\mathbb{Z}}

\newcommand{\Hom}{\mbox{Hom}}

\newcommand{\Alie}{(A\to M,[\cdot,\cdot],\rho)}

\DeclareMathOperator{\id}{Id} \DeclareMathOperator{\im}{im}

 \DeclareMathOperator{\End}{End}

\begin{document}
\sloppy
\title{The geometry of graded cotangent bundles}

\author{Miquel Cueca}
\address{Instituto de Matem\'atica Pura e Aplicada,
Estrada Dona Castorina 110, Rio de Janeiro, 22460-320, Brazil }
\email{micueten@impa.br}
\thanks{\textbf{\textit{Keywords:}} Graded manifolds, Higher Courant algebroids, Higher Dirac structures, AKSZ} 
\subjclass[2010]{58D50, 53D17, 53D18}

\begin{abstract}
Given a vector bundle $A\to M$ we study the geometry of the graded manifolds $T^*[k]A[1]$, including their canonical symplectic structures, compatible $Q$-structures and Lagrangian $Q$-submanifolds. We relate these graded objects to classical structures, such as higher Courant algebroids on $A\oplus\bigwedge^{k-1}A^*$ and higher Dirac structures therein, semi-direct products of Lie algebroid structures on $A$ with their coadjoint representations up to homotopy, and  branes on certain AKSZ $\sigma$-models.
\end{abstract}

\maketitle
\tableofcontents

\section{Introduction}

The present work is devoted to the study of the graded manifolds $T^*[k]A[1]$, where $A$ is a vector bundle over a manifold $M$, in terms of classical geometric objects. When $A\to M$ carries the structure of a Lie algebroid, these shifted cotangent bundles naturally arise in connection with different aspects
of Poisson geometry and related higher structures. For example, (i) they codify higher Courant brackets, (ii) they encode (shifted) semi-direct products of the Lie algebroid with their coadjoint representation up to homotopy, and (iii) they can be regarded as targets for AKSZ $\sigma$-models. We will explore these three directions in this paper.

To explain the relation with Courant algebroids, let us recall that, 
in the $1990's$,  Courant and Weinstein \cite{cou:bey} introduced the standard Courant algebroid $TM\oplus T^*M$ over a manifold $M$, and showed that Dirac structures therein (i.e., maximal isotropic, involutive subbundles) gave a unified perspective to various classical geometric structures. The connection with graded geometry was subsequently established by  Roytenberg, \v{S}evera and Vaintrob, who realized that the geometry of the standard Courant algebroid is codified by the degree $2$ symplectic $Q$-manifold
$T^*[2]T[1]M$, see \cite{roy:on, sev:some}; moreover, in this picture Dirac structures in $TM\oplus T^*M$ correspond to lagrangian $Q$-submanifolds of $T^*[2]T[1]M$.

More recently, various authors have considered ``higher'' versions of the  standard Courant algebroid on the vector bundle $TM\oplus \bigwedge^k T^*M$ (see e.g. \cite{hag:nam,hit:gen}), given by a $\bigwedge\nolimits^{k-1}T^*M$-pairing and a bracket on its space of sections as follows:
\begin{equation*}
\begin{array}{ll}
\langle X+\omega, Y+\tau\rangle=i_X \tau+i_Y\omega,\\
\cbrack{X+\omega}{ Y+\tau}=[X,Y]+\cL_X\tau-i_Yd\omega.
\end{array}
\end{equation*}
More generally, for any Lie algebroid $A\to M$, one can consider similar  structures on $A \oplus \bigwedge^k A^*$. Many proposals have appeared for possible analogues of Dirac structures in this context, see e.g. \cite{bi:dir, bur:hig, hag:nam, wad:nam, zam:inf}.

In this article, we follow \cite{roy:on} and use the graded symplectic manifolds $T^*[k]A[1]$ to describe the  higher Courant structures on $A\oplus \bigwedge^{k-1}A^*$, in such a way that the symplectic structure on $T^*[k]A[1]$ corresponds to the pairing on  $A\oplus \bigwedge^{k-1}A^*$.  This correspondence relies on the geometric description of (non-negatively) graded manifolds in terms of vector bundles given in \cite{bur:frob}, see Section~\ref{S2}.  The relation between $T^*[k]A[1]$ and the geometry of $A\oplus \bigwedge^{k-1}A^*$ has been observed in previous works \cite{bou:aks, zam:inf}, and here we give the precise correspondence.
 
Symplectic $Q$-structures on $T^*[k]A[1]$ (i.e., degree $1$ homological symplectic vector fields) are always
determined by degree $k+1$ functions $\theta$ on $T^*[k]A[1]$ satisfying the classical master equation $\{\theta,\theta\}=0$.
 Using the derived bracket construction (see e.g. \cite{roy:on}), we see that such functions give rise to brackets on sections of $A\oplus \bigwedge^{k-1}A^*$ that are compatible with the pairing. For $k=2$, these functions  have been 
 geometrically classified in \cite{kos:qua}. For $k=3$ we present a similar result in Theorem~\ref{Q=3}, see also \cite{gru:h-t, ike:pq3}, and for $k>3$ we show that such functions have a particularly simple description: they are the same as a Lie algebroid structure on $A\to M$ together with $H\in\Gamma\bigwedge^{k+1}A^*$ such that $d_A H=0$, see Theorem \ref{Q>3}. In this last case, the bracket on $A\oplus \bigwedge^{k-1}A^*$ is given by the usual formula:
\begin{equation*}
\llbracket a+\omega, b+\eta\rrbracket_H=[a,b]+\cL_a\eta-i_bd\omega-i_bi_aH.
\end{equation*}

Once compatible  $Q$-structures on $T^*[k]A[1]$ are understood, we define higher Dirac structures on $(A\oplus \bigwedge^{k-1}A^*,\langle\cdot,\cdot\rangle,\llbracket\cdot,\cdot\rrbracket_H)$ as the lagrangian $Q$-submanifolds of $(T^*[k]A[1],\{\cdot,\cdot\},\theta_H)$, and we provide their classical description in Corollaries \ref{charLagrangian} and \ref{charLag}, and Theorem \ref{Nan-Di-Qlag}.
Just as for ordinary Dirac structures, these higher Dirac structures encode interesting geometric objects. Examples include Nambu tensors \cite{duf:poi,tak:nam,wad:nam}, $k$-plectic structures \cite{can:on,rog:inf} and foliations. In the particular case when $A=TM$, $H=0$, and the higher Dirac structure has support on the whole $M$, we recover the notion of a (regular) Nambu-Dirac structure as defined by Hagiwara in \cite{hag:nam}. See Sections \ref{S3} and \ref{S4} for details.

If $\mathfrak{g}$ is a Lie algebra, recall that its cotangent bundle $T^*\mathfrak{g}$ has a Lie algebroid structure coming from its identification with the semi-direct product $\mathfrak{g}\ltimes \mathfrak{g}^*$ with respect to the coadjoint representation. We show in Section \ref{S5} that this picture extends to graded cotangent bundles $T^*[k]A[1]$ of Lie algebroids, but it involves a more general notion of representation. Indeed, for a Lie algebroid $\Alie$, it is well known that, in order to make sense of the adjoint and coadjoint representations, one needs the notion of {\em representation up to homotopy}, as defined in \cite{cam:rep}. There are two viewpoints to what should be the semi-direct product of a Lie algebroid with a $2$-term representation up to homotopy: one leads to VB-algebroids, see \cite{gra:vb}, while the other gives $L_k$-algebroids (as in \cite{bon:on}), see \cite{chen:hig}. We relate these two perspectives and prove that the $Q$-manifolds $(T^*[k]A[1],Q=\{\theta,\cdot\})$ appear naturally as the semi-direct product of $A$ with its coadjoint representation up to homotopy. We also incorporate the $H$-twisted case in this context.

In regard to $\sigma$-models, recall that the seminal work \cite{AKSZ} describes a procedure, referred to as ``AKSZ'', to create Topological Field Theories using the machinery of supermanifolds and the Batalin-Vilkovisky formalism. In \cite{roy:AKSZ} Roytenberg describes a graded refinement of the AKSZ procedure in which the space of fields is given by
\begin{equation*}
\mathrm{Maps} (T[1]\Sigma^{n+1},\cM),
\end{equation*}
where $\Sigma^{n+1}$ is a manifold of dimension $n+1$ and $\cM$ is a degree $n$ symplectic $Q$-manifold. If the manifold $\Sigma$ has boundary, one needs to impose additional conditions on the fields, and one possibility is given by {\em branes}; in this case, the boundary is required to take values in a lagrangian $Q$-submanifold of $\cM$. It follows that the symplectic $Q$-manifolds $(T^*[k]A[1],\{\cdot,\cdot\},\theta_H)$ produce AKSZ $\sigma$-models, and the study of their lagrangian $Q$-submanifolds is important to determine possible boundary conditions for the fields; see Section \ref{S6}. The case when $A\to M$ is a Lie algebra is known as {\em BF-theory} and was deeply studied  e.g. in \cite{cat:top, cat:loop, cat:hig}.

At the end of this article, we include general remarks and speculations concerning the integration of the symplectic $Q$-manifolds $(T^*[k]A[1],\{\cdot,\cdot\},\theta_H)$ by symplectic $n$-groupoids and their lagrangian $Q$-submanifolds by lagrangian $n$-subgroupoids, following recent ideas of \v{S}evera \cite{sev:int}.

\subsection*{Acknowledgements}
I thank Henrique Bursztyn and Rajan Mehta for fruitful discussions, helpful advice and for comments and suggestions on previous drafts of this paper. This work also profited from enlightening conversations with Alejandro Cabrera and Marco Zambon as well as all the symplectic community of Rio de Janeiro. This research was supported by a Ph.D. grant given by CNPq.

\section{Graded manifolds}\label{S2}
In this section we recall the basic notions of (non-negatively graded) graded manifolds, such as submanifolds, vector fields and graded Poisson structures, see e.g. \cite{cat:int} for more details. Since our objective is to study the graded manifolds $T^*[k]A[1]$ in terms of classical geometric objects, we recall the correspondence established in \cite{bur:frob} describing graded manifolds in terms of vector bundle data.  
    
\subsection{Basic definitions}
    A \emph{graded manifold of degree $n$} (or simply \emph{$n$-manifold}) is a pair $\cM=(M, C_\cM)$ where $M$ is a smooth manifold and $C_\cM$ is a sheaf of graded commutative algebras such that $\forall p\in M,\quad \exists \ U$ open around $p$ such that
    \begin{equation*}
        (C_\cM)_{|U}\cong C^{\infty}(M)_{|U}\otimes Sym V
    \end{equation*}
    as sheaves of graded commutative algebras and where $Sym V$ denotes the graded symmetric vector space of a graded vector space $V=\oplus_{i=1}^n V_i$. The manifold $M$ is known as the \emph{body} of $\cM$ and $f\in C^k_\cM$ is called a \emph{homogeneous function of degree $k$}; we write $|f|=k$ for the degree of $f$. We define the  total dimension of $\cM$ as
    \begin{equation}
    	totdim\cM =\dim M+\dim V=\dim M+\sum_{i=1}^n \dim V_i.
    \end{equation}
    A morphism between graded manifolds $\Psi:\cM\to \cN$ is a pair
    \begin{equation*}
        \left\{
        \begin{array}{ll}
            \psi:M\to N & \text{smooth map}, \\
            \psi^\sharp:C_\cN\to\psi_* C_\cM & \text{degree preserving morphism of sheaves of algebras over } N.
        \end{array}
        \right.
    \end{equation*}
    With these definitions, graded manifolds with their morphisms form a category.
    
    \begin{example}
    Given a vector bundle $A\to M$ we can define the $1$-manifold 
    \begin{equation}
     A[1]=(M,\Gamma\bigwedge\nolimits^\bullet A^*).
    \end{equation}
    where we give the sheaf in terms of their global sections. In this case, any vector bundle morphism induces a morphism between their corresponding graded manifolds. The functor shifting by $1$, $A\to A[1]$, is an equivalence between the category of vector bundles and the category of degree $1$ manifolds.
    \end{example}
    
    A \emph{degree $n$ algebra bundle} is a pair $(E\to M, m)$ where $E=\oplus_{i=1}^n E_i\to M$ is a graded vector bundle and $m: \Gamma E\otimes \Gamma E\to \Gamma E$ is a graded algebra structure on $\Gamma E\to M$, i.e. $e_i\in\Gamma E_i, \ e_j\in\Gamma E_j, m(e_i,e_j)\in\Gamma E_{i+j}$.
    
    \begin{example}
    Given a graded vector bundle $F=\oplus_{i=1}^{n}F_i\to M$, denote by $(Sym F)^{\leq n}\to M$ the graded vector bundle
	\begin{equation*}
	{(Sym F)^{\leq n}}_k=\left\{\begin{array}{ll}
	Sym^k F& 1\leq k\leq n,\\
	0& \text{ otherwise}.
	\end{array}\right.
\end{equation*}

For any homogeneous $\omega,\eta\in\Gamma (Sym F)^{\leq n}$ with degrees $k$ and $l$ we denote by $\omega\cdot\eta$ the operation that is the graded symmetric product if $k+l\leq n$ and $0$ otherwise. Therefore, $( (Sym F)^{\leq n}\to M, \cdot)$ is a degree $n$ algebra bundle.	
    \end{example}
    
\begin{definition}
     We say that a degree $n$ algebra bundle, $(E,m)$, is \emph{admissible} if there exists $F=\oplus_{i=1}^n F_i\to M$ and a graded vector bundle isomorphism
    \begin{equation*}
        \phi: E \to (Sym F)^{\leq n}
    \end{equation*}
    over the identity which is also a graded algebra isomorphism between $(E, m)$ and $((Sym F)^{\leq n},\cdot)$.
    \end{definition}
        
     Let $\cM=(M,C_\cM)$ be a degree $n$-manifold. By the local condition of the sheaf we have that
    \begin{equation}
        \forall i\in\{1,\cdots, n\}, \ \exists \ E_i\to M \text{ vector bundle such that } C^i_\cM=\Gamma E_i
    \end{equation}
    and the mutiplication of functions on $C_\cM$ gives rise to a graded algebra structure 
    $$m:\Gamma E\otimes \Gamma E\to\Gamma E$$
    on the graded vector bundle $E=\oplus_{i=1}^nE_i\to M$. One can also see that it is admissible, for more details see \cite{bur:frob}.
    
    \begin{remark}
    If we dualize admissible algebra bundles we obtain admissible coalgebra bundles and they naturally define a category. It is proved in \cite{bur:frob} that the category of admissible degree $n$ coalgebra bundle is equivalent to the category of degree $n$-manifolds. For the purpose of this article, it suffices to work with algebra bundles.
    \end{remark}
  
Let $(M,C_\cM)$ be a graded manifold. We say that $\cI\subseteq C_\cM$ is a subsheaf of \emph{homogeneous ideals} if for all $U$ open of $M, \ \cI_{|U}$ is an ideal of ${C_\cM}_{|U}$, i.e. ${C_\cM}_{|U}\cdot \cI_{|U}\subseteq \cI_{|U}$ and $\forall f\in \cI_{|U}$ its homogeneous components belong also to $\cI_{|U}$.

Given a subsheaf of homogeneous ideals $\cI\subseteq C_\cM$ we can define a subset of $M$ by
\begin{equation*}
Z(\cI)=\{  x\in M \ | \quad f(x)=0\ \forall f\in\cI^0=\cI\cap C_\cM^0\}. 
\end{equation*}
In addition we say that a homogeneous ideal is \emph{regular} if $\forall p\in Z(\cI), \ \exists \ U\subseteq M$ open around $p$ where there are local coordinates $\{x^i, y^j\}$ for which $\cI_{|U}=\langle y^j\rangle$. In this case $Z(\cI)$ becomes a closed embedded submanifold of $M$.

Observe that given a graded manifold $(M, C_\cM)$ and a subsheaf of regular homogeneous ideals $\cI\subseteq C_\cM$, we define a new graded manifold by $\cN=(N,C_{\cM}/\cI)$, where $N=Z(\cI)$. The graded manifold $\cN$ has a natural morphism to $\cM$ that we denote by $i:\cN\to \cM$.
  
  	Given $\cM=(M, C_\cM)$ a graded manifold we define a \emph{closed submanifold} as  a subsheaf of regular homogeneous ideals $\cI\subseteq C_\cM$. For simplicity, unless stated otherwise, we refer to subsheaves of regular homogeneous ideals just as ideals and submanifolds will be considered to be closed, although most of the results are valid in general.

\begin{proposition}[see \cite{bur:frob}]\label{sub=ideals}
	Let $\cM=(M, C_\cM)$ be a graded manifold with corresponding algebra bundle $(E=\oplus E_i, m)$. There is a 1-1 correspondence between:
	\begin{itemize}
		\item Submanifolds of $\cM$.
		\item Graded Subbundles $F=\oplus F_i\subseteq E$ over submanifolds such that 
		\begin{equation}\label{submanifoldeq}
		F_{k}\cap\bigoplus_{i+j=k} m(E_i, E_j)= \bigoplus_{i+j=k}m(F_i, E_j).
		\end{equation}
	\end{itemize}
\end{proposition}

\subsection{Vector fields}
Another aspect of graded manifolds that we need are vector fields. Let $\cM=(M,C_\cM)$ be a graded manifold. We define a \emph{vector field of degree $k$}, denoted by  $X\in\fX^k(\cM)$ and by $|X|$ its degree, as a degree $k$ derivation of the graded algebra $C_\cM$, i.e
	\begin{equation*}
		X:C_\cM^i\to C_\cM^{i+k} \quad \text{such that } \ X(fg)=X(f)g+(-1)^{|f|k}fX(g).
	\end{equation*}
As it happens in usual geometry, the graded commutator of vector fields is again a vector field, so $(\fX^\bullet(\cM),[\cdot,\cdot])$ forms a graded Lie algebra.

A \emph{$Q$-manifold} is a pair $(\cM, Q)$ where $\cM$ is a graded manifold and $Q\in\fX^{1}(\cM)$ such that 
	\begin{equation}\label{Q-man}
		[Q,Q]=2Q^2=0.
	\end{equation}
Given two $Q$-manifolds, a \emph{$Q$-morphism} is a graded manifold morphism for which the vector fields are related. With this notion $Q$-manifolds also form a category.  The following classical result, due to Vaintrob, exemplifies how vector fields on graded manifolds codify interesting information.

\begin{theorem}[see \cite{vai:lie}]\label{vaintrobcorr}
 There is an equivalence of categories between:
 
 \begin{itemize}
 	\item Lie algebroids $\Alie$.
 	\item Degree 1 $Q$-manifolds $(\cM,Q)$.
 \end{itemize}
\end{theorem} 

The equivalence is as follows: A Lie algebroid $\Alie$ defines the degree $1$ manifold $A[1]=(M, C_{A[1]}=\Gamma\bigwedge^\bullet A^*)$ and the vector field $Q=d_A$ is just the Lie algebroid differential, determined by: 
\begin{equation}\label{dif} 
d_A f=\rho^*df,\quad d_A\alpha(a,b)=\cL_{\rho(a)}\alpha-\cL_{\rho(b)}\alpha-\alpha([a,b]),\quad f\in C^\infty(M), \ \alpha\in\Gamma A^*, \ a,b\in\Gamma A.
\end{equation}
and extended as a derivation of the wedge product.

	Let $(\cM, Q)$ be a $Q$-manifold and $\cN\subseteq\cM$ a submanifold with associated ideal $\cI$. We say that $\cN$ is a \emph{$Q$-submanifold} if 
	\begin{equation}
		Q(\cI)\subseteq \cI.
	\end{equation}		
In other words, the preceding equation means that $Q$ is tangent to the submanifold, and $(\cN, Q_{|\cN})$ becomes a $Q$-manifold.

Given a vector bundle $A\to M$, recall that its \emph{Atiyah algebroid} is a Lie algebroid, that we denote by $\bA_A\to M$, whose sections are the infinitesimal automorphism of $\Gamma A$, i.e. pairs $(D,\sigma)$ where $\sigma$ is a vector field on $M$ and $D:\Gamma A\to \Gamma A$ satisfying
\begin{equation}\label{der}
 D(fa)=fD(a)+\sigma(f)a.
\end{equation}

Just to end with this quick review of vector fields we also need their geometric characterization for manifolds of degree $1$.
\begin{proposition}\label{vfdeg1}
		Let $A\to M$ be a vector bundle and consider $A[1]$. Then:
		\begin{itemize}
			\item $\fX^{-1}(A[1])\simeq \Gamma A$ and $\fX^0 (A[1])\simeq\Gamma \bA_A$ as a $C^\infty(M)$-module.
			\item $\fX^\bullet(A[1])=\langle \fX^{-1}(A[1]), \fX^0(A[1])\rangle_{C_{A[1]}}$.
		\end{itemize}
		
	\end{proposition}
	\begin{proof}
		As a sheaf of algebras, $C_{A[1]}$ is generated by $C^0_{A[1]}=C^\infty(M)$ and $C^1_{A[1]}=\Gamma A^*$ and derivations are completely determined by their actions on generators. Then for $X\in\fX^{-1}(A[1])$ we have that it sends $C^\infty(M)$ to zero and $\Gamma A^*$ to $C^\infty(M)$, and it must be linear. Therefore $X\in\fX^{-1}(A[1])\Leftrightarrow X\in\Gamma A$.
		
	Any $X\in\fX^{0}(A[1])$ is characterized by two maps $D:\Gamma A^*\to \Gamma A^*$ and $\sigma:C^\infty(M)\to C^\infty(M)$ and since is a vector field we know that $\sigma$ is a vector field on $M$ and satisfy \eqref{der}. Therefore $X\in\fX^0(A[1])\Leftrightarrow (D,\sigma)\in\Gamma \bA_A$.
	
	Now since $C_{A[1]}$ as a sheaf is locally generated by $\{x_i,\alpha^j\}$, where $\{x^i\}$ are coordinates on $M$ and $\{\alpha^j\}$ base of local sections of $\Gamma A^*$, we have that $\fX(A[1])$ is locally generated by $\frac{\partial}{\partial x_i}$ and $\frac{\partial}{\partial \alpha^j}=e_j$ as a module over $C_{A[1]}$, so we have the result.
	\end{proof}

\subsection{Graded Poisson structures}
The last general definition that we need is that of a graded Poisson manifold, in particular the symplectic case. Consider $\cM=(M,C_\cM)$ a graded manifold. We say that $\{\cdot,\cdot\}$ is a  \emph{degreee k Poisson structure} on $\cM$ if $(C_\cM,\{\cdot,\cdot\})$ is a sheaf of graded Poisson algebras, i.e. $\{\cdot,\cdot\}:C^i_\cM\times C^j_\cM\to C^{i+j+k}_\cM$ satisfies:
	\begin{enumerate}
		\item $\{f,g\}=-(-1)^{(|f|+k)(|g|+k)}\{g,f\}$.
		\item $\{f,gh\}=\{f,g\}h+(-1)^{(|f|+k)|g|}g\{f,h\}$.
		\item $\{f,\{g,h\}\}=\{\{f,g\},h\}+(-1)^{(|f|+k)(|g|+k)}\{g,\{f,h\}\}$.
	\end{enumerate}	 
	
	In addition we say that the Poisson structure is \emph{symplectic}, i.e. non-degenerate, if $\forall p\in M,\  \exists \ U$ open around $p$ with local coordinates $\{ x^i\}$ of $\cM$ such that the matrix $(\{x^i,x^j\}^0 (q))$, where $\{x^i,x^j\}^0$ denotes the degree $0$ component of the function $\{x^i, x^j \}$, is invertible for all $q\in U$.

	Let $Q\in\fX^{i}(\cM)$ be a vector field. We say that $Q$ is \emph{Poisson} if
	\begin{equation}\label{Q-poisson}
		Q(\{f,g\})=\{Q(f),g\}+(-1)^{(|f|+k)i}\{f,Q(g)\} \qquad \forall f,g\in C_\cM.
	\end{equation}
	
	An important result due to Roytenberg is that on a graded symplectic manifold, depending on degree, symplectic vector fields are necessarily hamiltonian.
	
	\begin{proposition}[see \cite{roy:on}]\label{sym=ham}
Let $(\cM,\{\cdot,\cdot\})$ be a graded symplectic manifold of degree $k$ and $Q\in \fX^{l}(\cM)$ a vector field satisfying \eqref{Q-poisson}. If $l>-k$ then $Q=\{\theta,\cdot\}$ for some $\theta\in C^{k+l}_\cM$
\end{proposition} 

	Finally, let $(\cM,\{\cdot,\cdot\})$ be a symplectic manifold and $\cN\subseteq \cM$ a submanifold with associated ideal $\cI$. We say that $\cN$ is \emph{coisotropic} if
	\begin{equation}\label{coieq}
		\{\cI,\cI\}\subseteq \cI.
	\end{equation}
	Observe that our definition of symplectic manifold allows examples where the $totdim$ is not even, but if it is, we say that $\cN$ is \emph{lagrangian} if it is coisotropic and $totdim \cN= \frac{1}{2}totdim \cM$.

\section{Graded Cotangent bundles}\label{S3}

With the basics of graded manifolds we can now introduce graded cotangent bundles. Consider a vector bundle $A\to M$. As we already saw we can define a degree $1$-manifold by $A[1]=(M, \Gamma\bigwedge^\bullet A^*)$. Just as in usual geometry, any graded manifold has associated tangent and cotangent bundles that are vector bundles over it. But now, vector bundles over graded manifolds can be shifted by an arbitrary degree, and in this way we define the graded manifolds $T^*[k]A[1], \ \forall k\in\bN^*$.

The case $k=1$ is a bit special, because new degree $0$ coordinates are introduced. In this case we take a completion of the polynomial coordinates of degree $0$ and obtain a degree $1$ manifold isomorphic to $T^*[1]A^*$.

For $k\geq 2$, when no new degree $0$ coordinates are added, $T^*[k]A[1]=(M,C_{T^*[k]A[1]})$ is a degree $k$ manifold with 
\begin{equation}\label{functions}
	C_{T^*[k]A[1]}= Sym\left( \fX^{\bullet+k}(A[1])\right),
\end{equation}
 where we consider the grading by total degree and $Sym$ denotes the graded symmetric product with $Sym^{0}(\fX^{\bullet+k}(A[1]))=C^\bullet_{A[1]}$ without the shifting by $k$.

In local coordinates $T^*[k]A[1]$ has the following expression: Take $U$ a trivialization for the bundle $A\to M$ and consider local coordinates $x^i$ on $U$ and $\alpha^j$ fibre coordinates for $A_{|U}$ (that we identify conveniently with sections of $A^*$), and denote by $a^j$ the dual coordinates for $A^*_{|U}$. The set \begin{equation*}
\{x^i, \alpha^j, a_j=\frac{\partial}{\partial \alpha^j}, \frac{\partial}{\partial x^i}\} \text{ with }  |x^i|=0, |\alpha^j|=1, |a_j|=k-1 \text{ and } |\frac{\partial}{\partial x^i}|=k
\end{equation*} are local coordinates for $T^*[k]A[1]$.

\subsection{Classical description of the manifolds $T^*[k]A[1]$ and their submanifolds}

Although we have a nice definition of the manifolds $T^*[k]A[1]$ in terms of sheaves over $M$ we also want to characterize them in terms of algebra bundles, because this point of view allows us to work with classical vector bundles.

\begin{proposition}\label{VBchar}
	Let $A\to M$ be a vector bundle. Then:
	\begin{itemize}
		\item For $k=1: \ T^*[1]A[1]=T^*[1]A^*=(A^*, \fX^\bullet(A^*))$. So $C_{T^*[1]A[1]}$ is generated by $C^\infty(A^*)$ and $\Gamma TA^*$.
		\item For $k=2:$ the algebra bundle corresponding to $T^*[2]A[1]$ is $(E=\oplus_{i=1}^2 E_i\to M, m)$, where
		\begin{equation}\label{VBchareq2}
		E_i=\left\{\begin{array}{ll}
		A\oplus A^* & i=1,\\
		\bA_{(A\oplus A^*,\langle,\rangle)}& i=2.
		
		\end{array}
		\right.\quad \begin{array}{l}
		m:\Gamma E_1\otimes \Gamma E_1\to \Gamma E_2 \\
		m(e_1, e_2)=e_1\wedge e_2
		\end{array}
		\end{equation}
		where $\bA_{(A\oplus A^*,\langle,\rangle)}$ denotes the Atiyah algebroid of derivations that preserve the natural pairing on $A\oplus A^*$ and $e_1\wedge e_2$ is seen as an orthogonal endomorphism of $A\oplus A^*$. 
		Therefore the generators of $C_{T^*[2]A[1]}$ are $C^0_{T^*[2]A[1]}=C^\infty(M)$ and $C^i _{T^*[2]A[1]}=\Gamma E_i$ for $i=1,2$.
		\item For $k> 2:$ the algebra bundle corresponding to $T^*[k]A[1]$ is $(E=\oplus_{i=1}^k E_i\to M, m)$, where
		\begin{equation}\label{VBchareq}
		E_i=\left\{\begin{array}{lll}
		\bigwedge\nolimits^i A^* & 1\leq i< k-1,&m:\Gamma E_i\otimes \Gamma E_j\to \Gamma E_{i+j}\\
		A\oplus\bigwedge\nolimits^{k-1} A^*& i=k-1,&m(\alpha_i, \alpha_j)=\alpha_i\wedge \alpha_j, \ i+j\leq k \text{ and } j\neq k-1\\
		\bA_{A}\oplus\bigwedge\nolimits^kA^*& i=k.&m( \alpha_1, a+\alpha_{k-1})=\alpha_1\otimes a+\alpha_1\wedge\alpha_{k-1}
		
		\end{array}
		\right.
		\end{equation}
		where $\alpha_1\otimes a$ is seen as an endomorphism inside $\bA_{A}$. 
		Therefore the generators of $C_{T^*[k]A[1]}$ are $C^0_{T^*[k]A[1]}=C^\infty(M)$ and $C^i _{T^*[k]A[1]}=\Gamma E_i$ for $i=1,\cdots,k$.
	\end{itemize}	 
\end{proposition}

\begin{proof}
	 For the case $k=1$ we already mentioned that we consider a completion of multi-vector fields over $A[1]$ shifted by $1$, so $T^*[1]A[1]$ is isomorphic to $T^*[1]A^*$.
	 
	 For $k=2$ equation \eqref{functions} combined with the geometric characterization of vector fields given by Proposition \ref{vfdeg1} tell us that:
	 \begin{itemize}
	 	\item $C^0_{T^*[2]A[1]}= C_{A[1]}^0 = C^\infty(M)$.
	 	\item $C^1_{T^*[2]A[1]}= C_{A[1]}^1\oplus \fX^{-1}(A[1])=\Gamma A^*\oplus \Gamma A=\Gamma E_1$.
	 	\item $C_{T^*[2]A[1]}^2 = C^2_{A[1]}\oplus \fX^{0}(A[1])\oplus \bigwedge^2 \fX^{-1}(A[1])=\Gamma(\bigwedge^2 A^*\oplus \bA_A\oplus \bigwedge^2 A)=\Gamma\bA_{(A						\oplus A^*,\langle,\rangle)}=\Gamma E_2.$
	 \end{itemize}
	  
	 The case $k> 2$ is analogous: equation \eqref{functions} and Proposition \ref{vfdeg1} tell us that
	 \begin{itemize}
	 	\item $C^0_{T^*[k]A[1]}= C_{A[1]}^0 = C^\infty(M)$.
	 	\item For $1\leq i< k-1,\quad C^i_{T^*[k]A[1]}= C_{A[1]}^i=\Gamma \bigwedge ^i A^*=\Gamma E_i$.
	 	\item $C_{T^*[k]A[1]}^{k-1} = \fX^{-1}(A[1]) \oplus C^{k-1}_{A[1]}=\Gamma(A\oplus \bigwedge^{k-1} A^*)=\Gamma E_{k-1}.$
	 	\item $C_{T^*[K]A[1]}^k= \fX^{-0}(A[1])\oplus C^k_{A[1]}=\Gamma(\bA_A\oplus \bigwedge^kA^*)=\Gamma E_k.$
	 \end{itemize}
	 and the algebra structure of $C_{T^*[k]A[1]}$ given by the multiplication of functions clearly gives the map $m$.
\end{proof}

\begin{remark}
Observe that \eqref{VBchareq2} is an admissible algebra bundle because if we define the graded vector bundle $\mathbf{A}=A_1\oplus A_2$, where $A_1=A\oplus A^*$, and $A_2=T^*M$, any connection on $A\oplus A^*$ that preserves the pairing induces an algebra isomorphism between \eqref{VBchareq2} and $((Sym \mathbf{A})^{\leq 2}, \cdot)$. Also \eqref{VBchareq} is an admissible algebra bundle because in this case any connection on $A$ induces an isomorphism between \eqref{VBchareq} and $((Sym \mathbf{A})^{\leq k}, \cdot)$ where \begin{equation*}
\mathbf{A}=\bigoplus_i A_i=\left\{\begin{array}{ll}
 A^*& i=1,\\
 A& i=k-1,\\
TM& i=k,\\
0& \text{otherwise}.
\end{array}\right.
\end{equation*}	
\end{remark}

From now on we focus on the case $k>2$. By the previous proposition, the case $k=1$ is just a vector bundle and it is well known what happens there. The case $k=2$ was described by Roytenberg in \cite{roy:on} and  corresponds to the standard structures on $A\oplus A^*$. Our results should be understood as a generalization of this case.

The following step is to obtain a classical description for the submanifolds of $T^*[k]A[1]$. For this we use Proposition \ref{sub=ideals} and obtain a characterization of the submanifolds in terms of the algebra bundles that we introduced in the previous proposition.

\begin{proposition}\label{cotsub}
For $k>2$ there is a 1-1 correspondence between:
\begin{itemize}
	\item Submanifolds of $T^*[k]A[1]$
	\item Quadruples $(N, D, L, K)$ where  $N\subseteq M$ is a submanifold and $D, L$ and $K$ are three vector bundles over $N$ satisfying
	\begin{eqnarray}
	&D\subseteq A^*_{|N},\  \ L\subseteq \left(A\oplus \bigwedge\nolimits^{k-1}A^*\right)_{|N} \text{ and } K\subseteq \left(\bA_A\oplus \bigwedge\nolimits^k A^*\right)_{|N},\label{Sub0}\\
		&L\cap\bigwedge\nolimits^{k-1}A^*_{|N}= D\wedge\bigwedge\nolimits^{k-2}A^*_{|N},\label{Sub1}\\
		&K\cap(\End(A)\oplus\bigwedge\nolimits^k A^*)_{|N}=D\otimes A_{|N}\oplus L\wedge A^*_{|N}.\label{Sub2}
	\end{eqnarray}
\end{itemize}
\end{proposition}

\begin{proof}
	By Proposition \ref{sub=ideals} we have that submanifolds are the same as a submanifold $N\subseteq M$ and  $F=\oplus_{i=1}^k F_i\to N$ vector subundle of \eqref{VBchareq} satisfying equation \eqref{submanifoldeq}. In our case this means that:
	
	\begin{enumerate}
		\item For $i=1:\  F_1=D\subseteq A^*_{|N}$ is vector subundle.
		\item  For $1< i< k-1: \ F_i= F_i\cap \wedge^{i} A^*_{|N}= D\cap \wedge^{i-1}A^*_{|N}$, so they are completely determined.
		\item For $i=k-1:\ F_{k-1}=L\subseteq( A\oplus \bigwedge^{k-1}A^*)_{|N}$ and equation \eqref{submanifoldeq} becomes \eqref{Sub1}.
		\item For $i=k: \ F_{k}=K\subseteq (\bA_A\oplus \bigwedge^{k}A^*)_{|N}$ and equation \eqref{submanifoldeq} becomes \eqref{Sub2}.
	\end{enumerate}
\end{proof}

\begin{remark}
 Denote by $p_1:A\oplus \bigwedge^{k-1}A^*\to A$ the projection. One of the consequences of equation \eqref{Sub1} is that $p_1(L)\subseteq A$ defines a vector subundle. This is not the case for $k=2$, and this property makes the submanifolds here much more rigid, as we will see.
\end{remark}

Given a submanifold $\cN\subseteq T^*[k]A[1]$ equivalent to $(N,D,L,K)$ denote  by $\widehat{F}$ the image of the projection of $K$ onto $TM$, that is a regular distribution. Then
\begin{equation}\label{totdimsub}
totdim \cN=\dim N+rk(A^*)-rk(D)+rk(A)-rk(p_1(L))+dim M-rk(\widehat{F}).
\end{equation}

Recall that given a graded manifold $\cM=(M, C_\cM)$, submanifolds are defined as sheaves of regular homogeneous ideals $\cI\subseteq C_\cM$. A direct consequence of the preceding proposition is that given a quadruple $(N, D, L, K)$ satisfying the hypothesis of Proposition \ref{cotsub}, the ideal $\cI\subseteq C_{T^*[k]A[1]}$ associated to the submanifold can be expressed in each degree as: 
\begin{equation*}
\begin{array}{c|ccccccc}\label{idealsub}
& 0& 1 & 2& & k-2&k-1&k\\
\hline
C_{T^*[k]A[1]}& C^\infty(M)&A^* &\bigwedge^2 A^*& \cdots &\bigwedge^{k-2}A^* &A\oplus\bigwedge^{k-1}A^*&\bA_A\oplus\bigwedge^k A^*\\
\cI&Z(N)&\widehat{\Gamma}D&\widehat{\Gamma}(D\wedge A^*_{|N})&&\widehat{\Gamma}(D\wedge \bigwedge^{k-3}A^*_{|N})&\widehat{\Gamma}L&\widehat{\Gamma}K\\
\end{array}
\end{equation*}
where $Z(N)=\{ f\in C^\infty(M) \ | \ f(n)=0 \ \forall n\in N\}$ and for any subbundle $(H\to N)\subseteq (G\to M)$ we denote
\begin{equation*}
\widehat{\Gamma}(H)=\{s\in\Gamma G\ | \ s(n)\in H \ \forall n\in N\}.
\end{equation*}

\subsection{The symplectic structure}
The other aspect of the graded cotangent bundles is that they are graded symplectic manifolds. As mentioned in the introduction the symplectic structure of $T^*[k]A[1]$ will be responsible for the pairing on $A\oplus\bigwedge^{k-1}A^*$, as it happens when $k=2$.

	The manifolds $T^*[k]A[1]$ are symplectic with a bracket of degree $-k$ on functions given by the graded version of the Schouten bracket of multi-vector fields for the manifold $A[1]$. By \eqref{functions} the functions on $T^*[k]A[1]$ are graded symmetric powers of vector fields on $A[1]$. We define the graded version of the Schouten bracket by the following rules:
	\begin{equation}\label{symbr}
		\{ X, f\}=X(f)\quad \text{and}\quad\{X,Y\}=[X,Y] \text{ for } f\in C_{A[1]}, \ X,Y\in\fX(A[1])
	\end{equation}
	extended by graded skew symmetry and as a derivation with respect to the graded symmetric product. With this definition, the Schouten bracket satisfies the graded analogues of Leibniz and Jacobi identities. If we shift the vector field coordinates by $k$ then it is clear that it defines a Poisson bracket of degree $-k$. It remains to see that it is symplectic. Notice that $\forall p\in M, \ \exists\ U$ open subset around $p$ such that $A_{|U}$ is trivial and $\{x^i, \alpha^j, a_j=\frac{\partial}{\partial \alpha^j}, \frac{\partial}{\partial x^i}\}$ defines local coordinates for $T^*[k]A[1]$. It is clear that this is a Darboux chart.

\begin{proposition}
For $k> 2$, the manifold $T^*[k]A[1]$ is equivalent to $(E,m)$ given in \eqref{VBchareq}. The Poisson bracket on $T^*[k]A[1]$ is determined by the following operations on $E=\oplus E_i\to M$:\begin{equation*}
	\begin{array}{ll}
		\langle\cdot,\cdot\rangle:\Gamma E_{k-1}\times\Gamma E_1\to C^\infty(M),&\langle a+\omega, \alpha\rangle=\alpha(a),\\
		\langle\cdot,\cdot\rangle:\Gamma E_{k-1}\times\Gamma E_{k-1}\to \Gamma E_{k-2},&\langle a_1+\omega_1, a_2+\omega_2\rangle=i_{a_1}\omega_2+i_{a_2}\omega_1,\\
		\cdot:\Gamma E_k \times C^\infty(M)\to C^\infty(M),&(D,\tau)\cdot f= X(f),\\
		\Psi:\Gamma E_{k}\times \Gamma E_1\to \Gamma E_{1},&\Psi_{(D,\tau)}\alpha=D(\alpha),\\
	\Upsilon:\Gamma E_{k}\times \Gamma E_{k-1}\to \Gamma E_{k-1},&\Upsilon_{(D,\tau)}(a+\omega)=D(a+\omega)-i_a\tau,	\\
	\left[\cdot,\cdot\right]:\Gamma E_{k}\times \Gamma E_{k}\to\Gamma E_{k},&
	\left[(D_1,\tau_1),(D_2,\tau_2)\right]=(\left[D_1,D_2\right],D_1(\tau_2)-D_2(\tau_1)).
	\end{array}
\end{equation*}
\end{proposition}
\begin{proof}
	Use that the Poisson bracket is defined by equations \eqref{symbr} and recall how vector fields act on functions and on other vector fields.
\end{proof}

Having described the symplectic structure, we now apply Proposition \ref{cotsub} to coisotropic and lagrangian submanifolds. We have special interest in describing the lagrangian ones because for $T^*[2]A[1]$ they are in correspondence with almost Dirac structures of $(A\oplus A^*, \langle\cdot,\cdot\rangle)$. In the degree $2$ case the coisotropic and lagrangian submanifolds of an arbitrary degree $2$ symplectic manifold where described in \cite{bur:super}. Our results must be compared with that ones.

\begin{theorem}[Coisotropic submanifolds]\label{coisochar}
For $k>2$, there is a 1-1 correspondence between:
	\begin{itemize}
		\item Coisotropic submanifolds of $T^*[k]A[1]$.
		\item Data $(N, D, L,\widehat{F}, \nabla)$, where:
			\begin{enumerate}
				\item $N\subseteq M$ is a submanifold.
				\item $D$ and $L$ are vector bundles over $N$ satisfying $D\subseteq A^*_{|N}$ and $L\subseteq (A\oplus \bigwedge^{k-1}A^*)_{|N}$ such that condition \eqref{Sub1} follows, $D\subseteq p_1(L)^\circ$ and $\langle L, L\rangle\subseteq D\wedge\bigwedge^{k-3} A^*_{|N}$.
				\item $\widehat{F}\subseteq TN$ is a regular and involutive distribution.
				\item $\nabla$ is flat partial $\widehat{F}$-connection in the vector bundle $\frac{p_1(L)^\circ}{D}$.
			\end{enumerate}
	\end{itemize}
\end{theorem}

\begin{proof}
 By Proposition \ref{cotsub} submanifolds are the same as $(N, D, L, K)$ satisfying equations \eqref{Sub0},\eqref{Sub1} and \eqref{Sub2}. Associated to the quadruple we have an ideal $\cI\subseteq C_{T^*[k]A[1]}$ given by $\cI^0=Z(N), \ \cI^1=\widehat{\Gamma} D,\ \cI^{k-1}=\widehat{\Gamma} L$ and $\cI^k=\widehat{\Gamma} K$. Denote the image of the projection of $K$ into $\bA_A$ by $\widetilde{F}$  and into $TM$ by $\widehat{F}$. Since $D, L$ and $K$ satisfy \eqref{Sub2} we know that $\widetilde{F}$ and $\widehat{F}$ are vector bundles. 
 
 Using the preceding description of the symplectic bracket in classical terms, the coisotropic condition \eqref{coieq} becomes the following equations:
 \begin{equation*}
 \begin{array}{lll}
 \{\cI^0,\cI^k\}\subseteq\cI^k&\Leftrightarrow&\widehat{F}\subset TN.\\
 \{\cI^1,\cI^{k-1}\}\subseteq\cI^0&\Leftrightarrow&D\subseteq p_1(L)^\circ.\\
 \{\cI^1,\cI^k\}\subseteq\cI^1&\Leftrightarrow&\widetilde{F} \text{ preserves } D.\\
 \{\cI^{k-1},\cI^{k-1}\}\subseteq\cI^{k-2}&\Leftrightarrow&\langle L, L\rangle\subseteq D\wedge \bigwedge^{k-3}A^*_{|N}.\\
 \{\cI^{k-1},\cI^{k}\}\subseteq\cI^{k-1}&\Leftrightarrow&\left\{\begin{array}{l}
  \widetilde{F} \text{ preserves } L.\\
  \text{The projection into the second factor of } K \text{ is inside } L\wedge A^*.
  \end{array}\right.\\
 \{\cI^k,\cI^k\}\subseteq\cI^k&\Leftrightarrow&K  \text{ is involutive so }\widetilde{F} \text{ and } \widehat{F} \text{ are involutive}.
 \end{array}
 \end{equation*}
 Since the second projection of $K$ is inside $L\wedge A^*$ and $K$ satisfies equation \eqref{Sub2} and preserves $L$ and $D$, we know that it is completely characterized by $\widehat{F}$ and a partial $\widehat{F}-$connection on the vector bundle $\frac{p_1(L)^\circ}{D}$. Moreover this connection must be flat since $\widetilde{F}$ is involutive.
\end{proof}

\begin{corollary}{(Lagrangian submanifolds)}\label{charLagrangian}
For $k>2$, a lagrangian submanifold of $T^*[k]A[1]$ is the same as a submanifold $N\subseteq M$ and  a vector bundle $L\to N$ such that $L\subseteq (A\oplus \bigwedge^{k-1}A^*)_{|N}$ and satisfying
\begin{eqnarray}
&p_1(L)\subseteq A \text{ is a subbundle,}\label{L0}\\
		& \ L\cap \bigwedge\nolimits^{k-1}A^*_{|N}=p_1(L)^\circ\wedge\bigwedge\nolimits^{k-2} A^*_{|N},\label{L2} \\
		&\langle L, L\rangle\subseteq p_1(L)^\circ\wedge\bigwedge\nolimits^{k-3}A^*_{|N}.\label{L1}
	\end{eqnarray}
\end{corollary}
\begin{proof}
	Since $totdim T^*[k]A[1]$ is even a lagrangian submanifold is a coisotropic one that has total dimension half of the total dimension of the manifold. Using the chacterization of coisotropic submanifolds of the Theorem \ref{coisochar} and formula \eqref{totdimsub} we obtain that a coisotropic submanifold given by $(N, D, L, \widehat{F}, \nabla)$ has total dimension equal to $1/2\ totdim T^*[k]A[1]$ if and only if $D=p_1(L)^\circ$ and $\widehat{F}=TN$. As a consequence, $\nabla$ is defined on a zero bundle so the only information that we need is a vector bundle $L\to N$ satisfying \eqref{L0}, \eqref{L2} and \eqref{L1}.
\end{proof}

The equations \eqref{L0}, \eqref{L2} and \eqref{L1} that define a lagrangian submanifold have also the following geometric interpretation:

\begin{corollary}\label{charLag}
For $k>2$, there is a 1-1 correspondence between:
\begin{itemize}
	\item Lagrangian submanifolds of $T^*[k]A[1]$.
	\item Pairs $(E\to N, \Omega)$ where $(E\to N)\subseteq(A_{|N}\to N)$ is a subbundle and $\Omega\in\Gamma\bigwedge^k E^*$.
\end{itemize}
\end{corollary}
\begin{proof}
Given a subbundle 
\begin{equation*}
	\begin{array}{lll}
		E&\xrightarrow{j}& A\\
		\downarrow&&\downarrow\\
		N&\xrightarrow{\widehat{j}}&M
	\end{array}
\end{equation*}
 the inclusion map induces a map $j^*:\Gamma A^*\to\Gamma E^*$ that extends to a map $j^*:\Gamma\bigwedge^iA^*\to\Gamma\bigwedge^iE^*$ for all $i\in\bN$. Recall that $ \frac{\bigwedge^{k-1}A^*_{|N}}{E^\circ\wedge\bigwedge^{k-2}A^*_{|N}}\cong\bigwedge^{k-1}E^*$ and the following diagram commutes:
\begin{equation*}
\begin{array}{lll}
\Gamma\bigwedge^{k-1}A^*_{|N}&\xrightarrow{j^*}&\Gamma\bigwedge^{k-1}E^*\\
i_e\ \downarrow&&\downarrow \ i_e\\
\Gamma\bigwedge^{k-2}A^*_{|N}&\xrightarrow{j^*}&\Gamma\bigwedge^{k-2}E^*
\end{array}
\end{equation*}
where $i_e$ is the contraction with $e\in \Gamma E$. 

Given a pair $(E\to N, \Omega)$,  define 
\begin{equation*}
L=\{ e+ w\in \Gamma(A_{|N}\oplus\bigwedge\nolimits^{k-1}A^*_{|N})|\ e\in\Gamma E \text{ and }  i_e\Omega=j^*w\}.
\end{equation*}
Clearly equation \eqref{L2} is satisfied because $p_1(L)=E$ and $L\cap\bigwedge^{k-1}A^*_{|N}=\ker j^* = E^\circ\wedge\bigwedge^{k-2}A^*_{|N}$ as we wish. Let us check now \eqref{L1}:
\begin{equation*}
j^*\langle e+w, e'+w'\rangle=j^*(i_e w'+i_{e'}w)=i_ej^*w'+i_{e'}j^*w=i_e i_{e'}\Omega+i_{e'}i_e\Omega=0
\end{equation*}
$\forall e+w, e'+w'\in \Gamma L$ and therefore we obtain that \eqref{L1} is also satisfied.

Given a lagrangian submanifold $(L\to N)\subseteq( A\oplus \bigwedge^{k-1}A^*)_{|N}$ define $E=p_1(L)$, which is a subbundle by hypothesis. The fact that $L$ satisfies equation \eqref{L2} means that it fits into the exact sequence
\begin{equation*}
\begin{array}{clclc}
\bigwedge\nolimits^{k-1}A^*&\to &A\oplus\bigwedge\nolimits^{k-1}A^*&\to& A\\
E^\circ\wedge\bigwedge\nolimits^{k-2} A^*_{|N}&\to& L&\to& E
\end{array}
\end{equation*}
so there exists a map $\varphi:E\to\frac{\bigwedge^{k-1}A^*_{|N}}{E^\circ\wedge\bigwedge^{k-2}A^*_{|N}}\cong\bigwedge^{k-1}E^*$ such that 
\begin{equation*}
L=\{e+w \in (A\oplus \bigwedge\nolimits^{k-1}A^*)_{|N} | \ e\in E ,\quad \varphi(e)=j^*w\}.
\end{equation*}
Since $L$ also satisfies equation \eqref{L1}, it follows that
\begin{equation*}
0=j^*\langle e+w, e+w'\rangle=j^*(i_e w'+i_{e'}w)=i_ej^*w'+i_{e'}j^*w=i_e\varphi(e')+i_{e'}\varphi(e)
\end{equation*}
$\forall e+w, \ e'+w'\in\Gamma L$ and this is equivalent to $\varphi=\Omega^\sharp$ for some $\Omega\in\Gamma\bigwedge^k E^*$.
\end{proof}

\begin{remark}
	As a consequence of Corollary \ref{charLag} we see that the condition $p_1(L)$ has constant rank is a bit restrictive. For some examples it is interesting to allow the case when $p_1(L)$ changes rank at points on the base. Observe that in this case the other equations also make sense.  We call \emph{weak lagrangian submanifold} a subbundle $L\to N$ over a submanifold of $M$ such that $L\subseteq (A\oplus\bigwedge^{k-1}A^*)_{|N}$ and \eqref{L2} and \eqref{L1} hold.

For these, the characterization of the Corollary \ref{charLag} is just true pointwise. In terms of sheaves of ideals, we are allowing non regular ideals in a particular way, but they are still closed for the Poisson bracket. When $M$ is just a point, the notion of weak lagrangians and lagrangians coincide.
\end{remark}

Let us compare our definition with previous ones that have appeared in the literature, in particular the one given by Hagiwara in \cite{hag:nam} when $A=TM$ and the one given by Wade in \cite{wad:nam}. 
\begin{definition}[see \cite{hag:nam}]
 An almost Nambu-Dirac structure in $TM\oplus \bigwedge^{k-1}T^*M$ is a subbundle $L\subseteq TM\oplus \bigwedge^{k-1}T^*M$ over $M$ satisfying 
	\begin{eqnarray}
			\bigwedge\nolimits^{k-1}p_1(L)=pr_2(L^\circ),\label{H2}\\
		(i_a'w+i_aw')_{|\bigwedge^{k-2}p_1(L)}=0& \forall a+w, a'+w'\in \Gamma L.\label{H1}
	\end{eqnarray}
	where $pr_2:T^*M\oplus \bigwedge^{k-1}TM\to \bigwedge^{k-1}TM$ is the projection. In addition, we say that $L$ is regular if $p_1(L)\subseteq TM$ is a subbundle. 
\end{definition}

\begin{theorem}\label{Na-Di-lag}
	There is a one to one correspondence between 
\begin{equation*}
\left\{
	\begin{array}{c}
	\text{ Regular almost Nambu-Dirac structures }\\
	\text{ in } \ TM\oplus \bigwedge^{k-1}T^*M
	\end{array}\right\}\leftrightharpoons\left\{
	\begin{array}{c}
	\text{ Lagrangian submanifolds }\\
	\text{ of } \ T^*[k]T[1]M \text{ with body } $M$
	\end{array}\right\}
\end{equation*}
More generally, almost Nambu-Dirac structures in $TM\oplus \bigwedge^{k-1}T^*M$ are in correspondence with weak lagrangian submanifolds of $T^*[k]T[1]M$ with body $M$. 	
\end{theorem}

\begin{proof}
	Using Corollary \ref{charLagrangian} we just need to prove that equations \eqref{L2} and \eqref{L1} are equivalent to \eqref{H2} and \eqref{H1}. Clearly \eqref{L1} is the same as \eqref{H1}. Now we check that the annhilator of \eqref{L2} gives \eqref{H2}:
	\begin{equation*}
	\left(p_1(L)^\circ\wedge\bigwedge\nolimits^{k-2}A^*_{|N}\right)^\circ=\bigwedge\nolimits^{k-1}p_1(L)
	\end{equation*}
	and
	\begin{equation*}
	\left(L\cap \bigwedge\nolimits^{k-1}A^*_{|N}\right)^\circ=pr_2\left(L^\circ\oplus(\bigwedge\nolimits^{k-1}A^*_{|N})^\circ\right)=pr_2(L^\circ\oplus A^*_{|N})=pr_2(L^\circ).
	\end{equation*}
	It is clear that if the almost Nambu-Dirac structure is regular we obtain a lagrangian submanifold otherwise we just have a weak lagrangian submanifold.
\end{proof}

\begin{definition}[see \cite{wad:nam}]
An almost Dirac structure of order $k$ on $A\to M$ is a subbundle $L\subseteq A\oplus \bigwedge^{k-1}A^*$ over $M$ satisfying that there exists $E\subseteq A$ subbundle over $M$ and $\Omega\in \Gamma\bigwedge^{k}E^*$ such that
\begin{eqnarray}
\forall p\in M\quad L_p=\{ e+w\in( A\oplus \bigwedge\nolimits^{k-1}A^*)_p| \ e\in E_p, \ i_e\Omega_p=w_{|\bigwedge^{k-1}E}\},\label{w1}\\
\forall Z_1,\cdots, Z_{k-1}\in\Gamma\bigwedge\nolimits^{k-1}E\quad \exists \ e\in\Gamma E \text{ such that } i_{Z_1}\Omega\wedge\cdots i_{Z_{k-1}}\Omega=i_e\Omega.\label{w2}
\end{eqnarray}

\end{definition}

Using Corollary \ref{charLag} we see that an almost Dirac structure of order $k$ as defined by Wade is the same as a lagrangian submanifold of $T[k]A[1]$ with body $M$ that in addition satisfies equation \eqref{w2}. Therefore the Wade almost Dirac structures of order $k$ are a particular case of lagrangian submanifolds of $T^*[k]A[1]$.

In the literature, apart from the definitions of Hagiwara and Wade there are at least four more approaches. Some of them are defined for an arbitrary vector bundle and some of them just for $TM$. By chronological order: In \cite{zab:bra} Bonelli and Zabzine define almost generalized Dirac structures as maximal isotropic subbundles of $TM\oplus\bigwedge^{k-1}T^*M$, in \cite{zam:inf} Zambon notices that the geometry of $TM\oplus\bigwedge^{k-1} T^*M$ is related to the one coming from $T^*[k]T[1]M$ but use different equations and call them higher Dirac structures. The work of Bi and Sheng \cite{bi:dir}, using the name of $(p,k)-$Dirac structure, gives a definition that includes the Nambu-Dirac structure of Hagiwara as the case $(k-1, k-2)$ and the higher Dirac of Zambon as the case $(k-1,0)$. Finally in \cite{bur:hig} Bursztyn, Martinez and Rubio give another possible definition  based on Zambon keeping the terminology, but with different requirements.

We believe that this Supergeometric approach sheds light on the Hagiwara definition, justifying his equations and putting them into the general framework of lagrangian submanifolds inside graded symplectic manifolds. 
 
\begin{example}[Conormal bundles]\label{foliations}
Given any submanifold of $A[1]$, i.e. $(B\to N)\subseteq (A\to M)$ subbundle, that we denote by $B[1]\subseteq A[1]$, the conormal bundle shifted by $k,\quad N^*[k]B[1]\subseteq T^*[k]A[1],$ defines a lagrangian submanifold. The associated vector bundle is given by
		\begin{equation*}
			L=B\oplus B^\circ\wedge \bigwedge\nolimits^{k-2}A^*_{|N}.
		\end{equation*}
The associated pair as in Corollary \ref{charLag} is $(B\to N, 0)$. In the case where $A=TM$ this is codifying regular distributions. Observe that in this case weak lagrangians are not encoding singular distributions because $L$ itself will be singular.
\end{example}   

\begin{example}[k-forms]\label{k-plectic}
	Using Corollary \ref{charLag} we know that given any $\omega\in\Gamma \bigwedge^kA^*$ we obtain a lagrangian submanifold by considering the pair $(A\to M, \omega)$. In fact, these lagrangian submanifolds correspond to the ones satisfying $$L\cap \bigwedge\nolimits^{k-1}A^*=0.$$
\end{example}

\begin{definition}
Let $A\to M$ be a vector bundle and $\Pi\in\Gamma\bigwedge^k A$. We say that $\Pi$ is \emph{decomposable} if $\forall p\in M$ where $\Pi_p\neq0$ there exists an open neighbourhood $U$ where $A_{|U}$ is trivial and basis of sections of $A, \ \{a_1, \cdots, a_k\}$ such that
\begin{equation*}
\Pi_{|U}=f(x) a_1\wedge\cdots\wedge a_k\qquad f(x)\in C^\infty(M).
\end{equation*}
\end{definition}

\begin{example}[Decomposable k-multi-vector field]\label{nambu}
Now we study the other extreme case. Suppose that we have a weak lagrangian submanifold over $M$ with $$L\cap A=0.$$ Then we know that $L=graph(\Lambda)$ for some  $\Lambda:S\to A$ with $S\subseteq \bigwedge^{k-1}A^*$ and satisfies 
\begin{eqnarray*}
&\ker(\Lambda)=\im(\Lambda)^\circ\wedge\bigwedge\nolimits^{k-2}A^*,\\
&i_{\Lambda(\omega)}\omega'+i_{\Lambda(\omega')}\omega\in\im(\Lambda)^\circ\wedge\bigwedge\nolimits^{k-3}A^*.
\end{eqnarray*}
If we impose that $S=\bigwedge^{k-1}A^*$ on the one hand the first equation says that $\Lambda=\Pi^{\sharp}$ for some $\Pi\in\Gamma\bigwedge^{k}A$. On the other hand, it was seen by Hagiwara on \cite{hag:nam} that the second equation is equivalent to $\Pi$ being decomposable. In general, decomposable tensors give rise to weak lagrangians, the condition that $p_1(L)$ is a subbundle implies that the tensor is zero or never vanishes, so we can dualize and obtain a $k$-form. This example is the main reason why we also want to consider weak lagrangian submanifolds.
\end{example}

\section{The symplectic $Q$-structure}\label{S4}

In the previous sections we established a relation between the symplectic structure on the manifolds $T^*[k]A[1]$ and the natural pairing on $A\oplus\bigwedge^{k-1}A^*$. For $k=2$, it was proved by \v{S}evera and Roytenberg that the Courant brackets on $A\oplus A^*$ correspond to the $Q$-structures on $T^*[2]A[1]$ that preserve the symplectic structure. Following this idea, we study the symplectic $Q$-structures on $T^*[k]A[1]$ and construct the corresponding brackets on $A\oplus\bigwedge^{k-1}A^*$. Finally we define higher Dirac structures on $A\oplus\bigwedge^{k-1}A^*$ as the Lagrangian $Q$-submanifolds of $T^*[k]A[1]$, these are called $\Lambda$-structures in \cite{sev:some}.

\subsection{Q-structures and twists}

As it was defined in Section \ref{S2} a symplectic $Q$-structure on $T^*[k]A[1]$ is a vector field $Q\in \fX^{1}(T^*[k]A[1])$ that satisfies equations \eqref{Q-man} and \eqref{Q-poisson}. For $k\geq 2$, we can use Proposition \ref{sym=ham} and conclude that $Q=\{\theta,\cdot\}$ for some $\theta\in C^{k+1}_{T^*[k]A[1]}$. In addition, we have the following equivalence for condition \eqref{Q-man}:
\begin{equation}\label{CME}
0=[Q,Q]=2Q^2=2\{\theta,\{\theta,\cdot\}\}\Leftrightarrow \{\theta,\theta\}=0,
\end{equation}
which follows from  the graded Jacobi identity of the symplectic bracket. This last equation is known as the classical master equation. Our goal is  to express $\theta\in C^{k+1}_{T^*[k]A[1]}$ satisfying \eqref{CME} in terms of
classical geometric objects.

The case $k=2$ was considered in detail by Kosmann-Schwarzbach, see e.g. \cite{kos:qua}, and it gives rise to Lie-quasi bialgebroids and quasi-Lie bialgebroids. As a consequence of the results that we will prove in this section, we will see that $k=2$ is the most flexible case, their higher versions being much more rigid. The next result deals with $k=3$, see \cite{gru:h-t,ike:pq3}:

\begin{theorem}\label{Q=3}
 Symplectic $Q$-structures on $T^*[3]A[1]$ are equivalent to $(A\to M, [\cdot,\cdot], \rho, \langle\cdot,\cdot\rangle, H)$ where $[\cdot,\cdot]:\Gamma (A\wedge A)\to \Gamma A$,  $\rho:A\to TM$ vector bundle map covering the identity, $\langle\cdot,\cdot\rangle$ is a pairing on the vector bundle $A^*$  and $H\in\Gamma\bigwedge^4 A^*$ satisfying the following conditions:
 \begin{equation*}
		\left\{\begin{array}{lr}
		\left[a,fb\right]=f\left[a,b\right]+\cL_{\rho(a)}f b &a,b,c\in\Gamma A,\\
		\left[a,\left[b,c\right]\right]+c.p.=\flat(i_ai_bi_c H)& f\in C^\infty(M),\\
		\langle\cL_a \omega,\tau\rangle+\langle \omega,\cL_a\tau\rangle=\rho(a)\langle \omega,\tau\rangle& \omega,\tau\in\Gamma A^*,\\
		d_A H=0.
		\end{array}
		\right.
	\end{equation*}
	where $\flat: A^*\to A$ is the morphism induced by the pairing on $A^*$ and $d_A:\Gamma\bigwedge^j A^*\to \Gamma\bigwedge^{j+1} A^*$ is the derivation associated with the pair $[\cdot,\cdot]$ and $\rho$, as defined in \eqref{dif}.
\end{theorem}

\begin{proof}
Since $Q=\{\lambda,\cdot\}$, it is enough to study degree $4$ functions on $T^*[3]A[1]$ that satisfy the classical master equation. Using \eqref{functions} we have that 
\begin{equation*}
C^{4}_{T^*[3]A[1]}=\fX^{1}(A[1])\oplus Sym^2\  \fX^{-1}(A[1])\oplus C^{4}_{A[1]}
\end{equation*} 
so any function of degree $4$ can be written as $\theta+\pi+H$ where $\theta\in \fX^{1}(A[1])$, $\pi\in Sym^2\ \fX^{-1}(A[1])=\Gamma Sym^2 A$ and $H\in C^{4}_{A[1]}=\Gamma\bigwedge^4 A^*$. By a slight generalization of Theorem \ref{vaintrobcorr}, it is easy to show that $\theta\in\fX^1(A[1])$ is equivalent to a bracket, $[\cdot,\cdot]:\Gamma(A\wedge A)\to\Gamma A$ and an anchor $\rho:A\to TM$ satisfying the first equation of the statement. Also, one can identify $\pi\in\Gamma Sym^2 A$ with a pairing on $A^*$.

Finally, the classical master equation is equivalent to 
\begin{equation*}
0=\{\theta+\pi+H,\theta+\pi+H\}=\{\theta,\theta\}+\{\pi,\pi\}+\{H,H\}+2(\{\theta,\pi\}+\{\theta,H\}+\{\pi,H\})
\end{equation*}
and if we remember that the symplectic bracket on $T^*[3]A[1]$ corresponds to the Schouten bracket on $A[1]$ shifted by $3$ and recall how vector fields act on other vector fields and on functions, we obtain that this equation is equivalent to
\begin{equation*}
\left\{\begin{array}{l}
\{\theta,\theta\}+2\{\pi,H\}=0,\\
\{\theta,\pi\}=0,\\
\{\theta,H\}=0.
\end{array}\right.
\end{equation*}
It is easy to check that these three equations are equivalent to the last three equations in the statement. 
\end{proof}

In particular, if the pairing is zero we have a honest Lie algebroid on $A\to M$ with $H\in\Gamma \bigwedge^4 A^*$ satisfying $d_A H=0$. On the other hand, when $H=0$ we obtain a Lie algebroid on $A\to M$ with an ad-invariant paring on $A^*\to M$. The case when the three structures are nontrivial is analogous to the case $k=2$ and should be thought of as a quasi-algebroid. When $A\to M$ is just a vector space these structures were described in \cite{cat:top}. For other examples see \cite{gru:h-t,ike:pq3}.

\begin{theorem}\label{Q>3}
For $k>3$, symplectic $Q$-structures on $T^*[k]A[1]$ are equivalent to a Lie algebroid structure on $A\to M$ and $H\in \Gamma\bigwedge^{k+1}A^*$ with $d_A H=0$.
\end{theorem}

\begin{proof}
As in the previous theorem, we must study the functions of degree $k+1$ satisfying the classical master equation. Since $k>3$, we have that 
\begin{equation*}
C^{k+1}_{T^*[k]A[1]}=\fX^1(A[1])\oplus C^{k+1}_{A[1]}
\end{equation*}
by \eqref{functions}, so any $k+1$ function can be written as $\theta+H$, where $\theta\in\fX^1(A[1])$ and $H\in C^{k+1}_{A[1]}=\Gamma\bigwedge^{k+1}A^*$; the classical master equation is equivalent to
\begin{equation*}
 \left\{ \begin{array}{l}
\{\theta,\theta\}=0,\\
\{\theta,H\}=0.
\end{array}
\right.
\end{equation*}

By Theorem \ref{vaintrobcorr}, $\theta$ defines a Lie algebroid on $A\to M$. The other equation says that $d_A H=0$, where $d_A$ is the Lie algebroid differential on $\Gamma\bigwedge^\bullet A^*$.
\end{proof}

From now on, we use the notation $\theta_H=\theta+H$ and $X_{\theta_H}=\{\theta_H,\cdot\}$ for the corresponding hamiltonian vector field.
\subsection{Q-cohomology}

If we have a $Q$-manifold $(\cM, Q)$, the fact that $Q$ has degree $1$ and $Q^2=0$ implies that $(C_{\cM}, Q)$ becomes a differential complex. Therefore, we have an associated cohomology that we denote by $H_Q(\cM)$. Observe that, in contrast with other well known cohomologies, it is usual that $H^n_Q(\cM)\neq 0 \quad \forall n\in\bN$. For example, if $Q=0$ and $\cM$ has degree $2$ then $H^n_Q(\cM)=C_\cM^n\neq 0$. For the standard Courant algebroid $(T^*[2]T[1]M,\{\cdot,\cdot\},\theta)$ this cohomology is isomorphic to the de Rham cohomology of $M$. 

Also it is well known that exact Courant algebroids are classified by their \v{S}evera classes $[H]\in H^3(M)$. We will extend this result for $k>3$.

\begin{definition}
For any $k\geq 2$ and any element $B\in\Gamma\bigwedge^{k}A^*=C^k_{A[1]}\subset C^k_{T^*[k]A[1]}$, let $\tau^{B}:T^*[k]A[1]\to T^*[k]A[1]$  be the graded automorphism given by the exponential of the hamiltonian vector field of $B$, i.e. $\tau^B=Id-\{B,\cdot\}$, or more concretely:
\begin{equation*}
	\tau^B=\left\{\begin{array}{l}
	\tau^{B}=Id: M\to M,\\
	\tau^B:C_{T^*[k]A[1]}\to C_{T^*[k]A[1]}, \quad \tau^B(f+X)=f+X+X(B).
	\end{array}
	\right.
\end{equation*}
for $ f\in C_{A[1]},\  X\in \fX(A[1])$.
\end{definition}

	The symmetries $\tau^B$ were introduced in \cite{gua:gen} and also defined in \cite{bi:on, zab:bra, bou:aks}.
\begin{proposition}\label{symetries}
The diffeomorphisms $\tau^B:T^*[k]A[1]\to T^*[k]A[1]$ are symplectomorphisms and send $X_{\theta_H}$ to $X_{\theta_{H+d_A B}}$.
\end{proposition}

\begin{proof}
	A diffeomorphism is a symplectomorphism if and only if preserves the Poisson bracket and we have that $\forall \ f,g\in C_{T^*[k]A[1]}$
	\begin{equation*}
		\{\tau^B f,\tau^B g\}=\{f-\{B,f\},g-\{B,g\}\}=\{f,g\}-\{f,\{B,g\}\}-\{\{B,f\},g\}=\{f,g\}-\{B,\{f,g\}\}
	\end{equation*}
	where we use the graded Jacobi identity and the fact that $\{\{B,f\},\{B,g\}\}=0\ \forall f,g\in C_{T^*[k]A[1]}$. Therefore $\tau^B$ are symplectomorphisms. For the second assertion, since the maps are symplectomorphisms,  it is enough to notice that
	\begin{equation*}
	\tau^B(\theta+H)=\theta+H+\theta(B)=\theta+H+d_A(B).\qedhere
	\end{equation*}
	
\end{proof}

As a consequence of Proposition \ref{symetries}, we have that if $(A\to M, [\cdot,\cdot],\rho)$ is a Lie algebroid and $H,H'\in\Gamma\bigwedge^{k+1}A^*$ are such that $d_A H=d_A H'=0$ with $H-H'=d_A B$ for some $B\in\Gamma\bigwedge^kA^*$, so $(T^*[k]A[1],X_{\theta_H})$ and $(T^*[k]A[1],X_{\theta_{H'}})$ have the same $Q$-cohomology.

\begin{corollary}
Let $(A\to M,[\cdot,\cdot],\rho)$ be a Lie algebroid and denote by $H^\bullet(A)$ its Lie algebroid cohomology. For $k>3$, the $Q$-structures on $T^*[k]A[1]$ compatible with the symplectic structure for which $p:T^*[k]A[1]\to A[1]$ is a $Q$-morphism are parametrized up to equivalence by $H^{k+1}(A)$.
\end{corollary}

\subsection{Brackets and higher Dirac structures}
Given a Lie algebroid $\Alie$ and $H\in\Gamma\bigwedge^{k+1}A^*$ with $d_AH=0$, consider the symplectic $Q$-manifold $(T^*[k]A[1],\{\cdot,\cdot\},\theta_H)$. We can use the derived bracket formalism with respect to the symplectic bracket and $\theta_H$ to obtain a bracket on sections of $A\oplus \bigwedge^{k-1}A^*$ by the following formula:
$$\{\{\cdot,\theta_H\},\cdot\}: C^{k-1}_{T^*[k]A[1]}\times C^{k-1}_{T^*[k]A[1]}\to C^{k-1}_{T^*[k]A[1]},$$
where we use formulas \eqref{functions} and \eqref{VBchareq} to identify 
$$C^{k-1}_{T^*[k]A[1]}=\fX^{-1}(A[1])\oplus C^{k-1}_{A[1]}=\Gamma(A\oplus \bigwedge\nolimits^{k-1}A^*).$$ 
Therefore, the above formula defines a bracket $\cbrack{\cdot}{\cdot}_H:\Gamma(A\oplus \bigwedge^{k-1}A^*)\times\Gamma(A\oplus \bigwedge^{k-1}A^*)\to \Gamma(A\oplus \bigwedge^{k-1}A^*)$ that in classical terms is just
\begin{equation}\label{Cbrack}
\begin{array}{lll}
\cbrack{a+\omega}{b+\eta}_H&=&\{\{\widehat{a}+\omega,\theta+H\},\widehat{b}+\eta\}=\{\left[\widehat{a},\theta\right]+\widehat{a}(H)+d_A\omega,\widehat{b}+\eta\}\\
&=&\left[\left[\widehat{a},\theta\right],\widehat{b}\right]+\left[\widehat{a},\theta\right](\eta)-\widehat{b}(\widehat{a}(H))-\widehat{b}(d_A\omega)\\
&=&\left[a,b\right]+\cL_a\eta-i_bd_A\omega-i_bi_aH,
\end{array}
\end{equation}
where $a,b\in \Gamma A,\ \omega,\eta\in\Gamma \bigwedge^{k-1}A^*$, and $\widehat{a},\widehat{b}\in\fX^{-1}(A[1])=\Gamma A$ denote the vector fields induced by $a$ and $b$, respectively.
\begin{remark}
For $A=TM$, these brackets have been previously  considered e.g. in \cite{gua:gen, hag:nam, hit:gen}. These works also explain the $H$-twist and the symmetries $\tau^{B}$. The connection with the graded manifolds $T^*[k]T[1]M$ was noticed in \cite{bou:aks, zam:inf}. 
\end{remark}

\begin{definition}
Let $\Alie$ be a Lie algebroid and $H\in\Gamma\bigwedge^{k+1}A^*$ with $d_A H=0$. For $k>2$, a \emph{higher Dirac} structure on $(A\oplus \bigwedge^{k-1}A^*, \langle\cdot, \cdot\rangle, \cbrack{\cdot}{\cdot}_H, \rho)$ is a vector subbundle $L\to N$ of $(A\oplus \bigwedge^{k-1}A^*)_{|N}$ over a submanifold $N\subseteq M$ satisfying equations \eqref{L0}, \eqref{L2}, \eqref{L1} and
 $$\rho(L)\subset TN\quad \text{and}\quad \cbrack{\Gamma L}{\Gamma L}_H\subseteq\Gamma L.$$
 $N$ is called the body or ``support'' of the higher Dirac structure ( See e.g. \cite{bur:sup} for Dirac structures with support in the ordinary sense).
\end{definition}

\begin{theorem}\label{Nan-Di-Qlag}
	Let $(A\to M,[\cdot,\cdot],\rho)$ be a Lie algebroid and $H\in\Gamma\bigwedge^{k+1}A^*$ with $d_A H=0$. For $k>2$, there is a one to one correspondence:
	\begin{equation*}
	\left\{\begin{array}{cc}
	\text{Higher Dirac structures}\\
	\text{on } (A\oplus \bigwedge\nolimits^{k-1}A^*, \langle\cdot,\cdot\rangle, \cbrack{\cdot}{\cdot}_H, \rho)
\end{array}	 \right\}\rightleftharpoons\left\{\begin{array}{cc}
\text{Lagrangian Q-submanifolds}\\
\text{of }(T^*[k]A[1],\{\cdot,\cdot\},X_{\theta_H})
\end{array}\right\}
	\end{equation*}
\end{theorem}

\begin{proof}
	By Corollary \ref{charLagrangian} the only thing that remains to prove is that $\rho(L)\subset TN$ and $\cbrack{\Gamma L}{\Gamma L}_H\subseteq \Gamma L$ if and only if $Q=X_{\theta_H}$ is tangent to the lagrangian submanifold. Denote by $\cI$ the sheaf of ideals that defines the lagrangian submanifold. Recall that on a lagrangian submanifold 
	\begin{equation*}
		 \ f\in\cI^i\Leftrightarrow X_f(\cI)\subseteq \cI\quad \forall  i\in\bN^*.
\end{equation*}		
Therefore	 
	\begin{equation*}
		\begin{array}{lll}
		Q(\cI)\subseteq \cI &\Leftrightarrow& X_{Q(f)}(\cI)\subseteq \cI \quad \forall f\in\cI\\
		&\Leftrightarrow& \{\{f,\theta_H\},g\}\in \cI\quad\forall f,g\in\cI\\
		&\Leftrightarrow&\rho(L)\subseteq TN \text{ and } \cbrack{\Gamma L}{\Gamma L}_H\subseteq \Gamma L.
		\end{array}\qedhere
	\end{equation*}		
\end{proof}

\begin{corollary}
Given a lagrangian submanifold of $(T^*[k]A[1],\{\cdot,\cdot\},\theta_H)$ equivalent to the pair $(E\to N, \Omega)$ then $X_{\theta_H}$ is tangent to it if and only if $E\to N$ is a subalgebroid of $A\to M$ and $d_E\Omega=j^*H$.
\end{corollary}
\begin{proof}
Denote by $j:E\to A$ the inclusion and recall that
\begin{equation*}
L=\{e+w\in \Gamma(A\oplus\bigwedge\nolimits^{k-1}A^*)_{|N}\ |\ e\in\Gamma E \text{ and } i_e\Omega=j^*w\}
\end{equation*}
with $\rho(E)\subseteq TN$  and  $\cbrack{\Gamma L}{\Gamma L}_H\subseteq \Gamma L$. The second condition implies that, $\forall e+w, \ e'+w'\in\Gamma L$,
\begin{equation*}
\cbrack{e+w}{e'+w'}_H=[e,e']+\cL_e w'-i_{e'}d_Aw-i_{e'}i_e H\in \Gamma L,
\end{equation*}
so $\rho(E)\subseteq TN$ and, $\forall e, e'\in\Gamma E$, $[e,e']\in\Gamma E$ if and only if $E\to N$ is a Lie subalgebroid. In particular this condition implies that $j^*d_A=d_Ej^*$, so $j^*$ is a chain map and therefore
\begin{equation*}
i_{[e,e']}\Omega=j^*(\cL_e w'-i_{e'}d_Aw-i_{e'}i_e H)\Leftrightarrow d_E\Omega= j^*H.\qedhere
\end{equation*}
\end{proof}

\begin{example} 
Consider the zero section $N^*[k]A[1]$, that by definition is a $Q$-lagrangian inside $(T^*[k]A[1],\{\cdot,\cdot\},X_\theta)$. Given $\omega\in\Gamma\bigwedge^{k}A^*$ with $d\omega= H$ we could take the symplectomorphism $\tau^{\omega}$ that sends $X_\theta$ to $X_{\theta_H}$ and $N^*[k]A[1]$ to $L=graph(\omega)$. Therefore $L$ is a $Q$-lagrangian submanifold inside $(T^*[k]A[1],\{\cdot,\cdot\},X_{\theta_H})$.
\end{example}

Recall that decomposable tensors, see Example \ref{nambu}, correspond to weak lagrangian submanifolds. To finish this section, we analyze their involutivity condition.

\begin{definition}
Let $A\to M$ be a Lie algebroid and $H\in\Gamma\bigwedge^{k+1}A^*$ with $d_A H=0$. Consider $\Pi\in\Gamma\bigwedge^k A$ a decomposable tensor. We say that it defines an \emph{$H-$twisted Nambu structure} if
\begin{equation}\label{int-nan}
		(\cL_{\Pi(w)}\Pi)(w')=-\Pi\big(i_{\Pi(w')}d_Aw+i_{\Pi(w')}i_{\Pi(w)}H\big)\qquad \forall w,w'\in\Gamma\bigwedge\nolimits^kA^*.
	\end{equation}
\end{definition}

	In \cite{wad:nam} Wade introduced this equation for general Lie algebroids and here we extend it with the $H$-twist. Wade also observed that the decomposability of the tensor does not follow from equation \eqref{int-nan}, so we must impose it in order to obtain a weak lagrangian submanifold.

\begin{example}[$H$-twisted Nambu tensor]
	Let $L=graph(\Pi)$ where $\Pi\in\Gamma\bigwedge^k A$ is a decomposable tensor. We claim that $L$ is a higher Dirac structure on  $(A\oplus\bigwedge^{k-1}A^*,\langle\cdot,\cdot\rangle,\cbrack{\cdot}{\cdot}_H)$ if and only if $\Pi$ is an $H-$twisted Nambu structure. The proof follows from the following computation for all $w,w'\in\Gamma\bigwedge^{k-1} A^*$:
	\begin{equation*}
	\begin{array}{lcl}
	[\Pi(w),\Pi(w')]&=&\cL_{\Pi(w)}\left(\Pi(w')\right)\\
	&=&\left(\cL_{\Pi(w)}\Pi\right)w'+\Pi\left(\cL_{\Pi(w)}w'\right)\\
	&=&-\Pi\big(i_{\Pi(w')}d_Aw+i_{\Pi(w')}i_{\Pi(w)}H\big)+\Pi\left(\cL_{\Pi(w)}w'\right)\\
	&=&\Pi\left(\cL_{\Pi(w)}w'-i_{\Pi(w')}d_Aw-i_{\Pi(w')}i_{\Pi(w)}H\right).
	\end{array}
	\end{equation*}  
\end{example}

\section{Semi-direct product of Lie algebroids with $2$-term representations up to homotopy}\label{S5}

We will see in this section that $T^*[k]A[1]$, when $A\to M$ is a Lie algebroid, has the structure of a semi-direct product, extending the well known Lie algebra isomorphism between $T^*\fg$ and $\fg\ltimes\fg^*$, where $\fg\ltimes\fg^*$ stands for the semi-direct product of $\fg$ with their coadjoint representation.

In order to understand the adjoint and coadjoint representations of a Lie algebroid, Abad$-$Crainic and Gracia-Saz$-$Mehta introduced the notion of $2$-term representation up to homotopy, see \cite{cam:rep, gra:vb}. In the literature, we find two different viewpoints to the semi-direct product of a Lie algebroid with a $2$-term representation up to homotopy: The first one leads to VB-algebroids, see \cite{gra:vb}; the second produces an $L_2$-algebroid or, more generally, an $L_k$-algebroid, see \cite{chen:hig}. 

The goal of this section is to explain the relation between these two viewpoints. We will do that by expressing both in terms of graded $Q$-manifolds, concluding that they coincide up to ``splittings''. The relation is
schematically illustrated in the following diagram:
\begin{equation}\label{diagram}
	\begin{array}{ccccc}
	\left\{\begin{array}{c}
	VB\text{-algebroid}\\
	\text{ over } A\to M
	\end{array}\right\}& \xlongrightarrow[F_1]{[1]}&	\left\{\begin{array}{c}
	\text{Deg. } 1\ Q\text{-bundle}\\
	\text{ over } A[1]
	\end{array}\right\}&\xrightarrow[F_2]{[k]}&	\left\{\begin{array}{c}
	\text{Deg. } k \ Q\text{-bundle}\\
	\text{ over } A[1]
	\end{array}\right\}\\
	\begin{array}{c}
	\\
	S_1\downarrow{Splitting}\\
	\\
	\end{array}&&&&S_2\downarrow{Splitting}\\
	\left\{\begin{array}{c}
	2\text{-term rep. up}\\
	\text{ to homotopy } 
	\end{array}\right\}&\xrightarrow[G_1]{\ltimes[k]}&	\left\{\begin{array}{c}
	L_k\text{-algebroid }\\
	\text{extension of } A
	\end{array}\right\}&\xrightarrow[G_2]{[1]}&	\left\{\begin{array}{c}
	\text{Split deg. } k \\
	Q\text{-bundle over } A[1]
	\end{array}\right\}
	\end{array}
\end{equation} 

Morally, the diagram says that the functor ``$[1]$" (arrows ``$F_1$" and ``$G_2$"), that goes from the category of $L_k$-algebroids to the category of $Q$-manifolds, commutes with the functor ``$[k]$" (arrows ``$F_2$" and ``$G_1$"), which shifts by $k$ $2$-term representations up to homotopy or the fibres of a $Q$-bundle.

The arrows ``$S_1$'' and ``$F_1$'' are explained in \cite{gra:vb}, while the arrow ``$G_1$'' for $k=2$ is constructed in \cite{chen:hig}, and the next arrow ``$G_2$'' is from \cite{bon:on}. We will briefly review these constructions in this section and explain the remaining arrows. This also must be compared with \cite{raj:mod} where Mehta construct an arrow going from $Q$-bundles over $A[1]$ to $n$-term representations up to homotopy. 

In the previous diagram, starting with the VB-algebroid $T^*A$ for a given Lie algebroid $(A\to M, [\cdot,\cdot],\rho)$, the corresponding graded $Q$-manifold on the top row is  the shifted cotangent bundle $(T^*[k]A[1],\{\cdot,\cdot\},\theta)$ that we studied in the previous sections. In this case, we also explain how to incorporate the $H$-twisted $Q$-structures of Theorem \ref{Q>3}.

\subsection{2-term representations up to homotopy}

Following \cite{cam:rep}, given a Lie algebroid $\Alie$ a \emph{2-term representation up to homotopy} is defined by:
\begin{enumerate}
	\item Two vector bundles $E_0,\ E_1$ over $M$ and a vector bundle map $\partial: E_0\to E_1$ over the identity.
	\item An $A-$connection on each vector bundle, $\nabla^0$ and $\nabla^1$, satisfying $\partial\circ\nabla^0=\nabla^1\circ\partial$.
	\item A two form $K\in\Gamma\left(\bigwedge^2 A^* \otimes \Hom(E_1,E_0)\right)$ such that:
	\begin{equation*}
		F^{\nabla^0}=K\circ \partial,\quad F^{\nabla^1}=\partial\circ K \quad \text{and} \quad d^\nabla(K)=0,
	\end{equation*}
	where $F^{\nabla^0},F^{\nabla^1}$ denote the curvature of the respective connection.
\end{enumerate}

With this definition and the choice of an auxiliary connection on the vector bundle $A\to M$, Abad and Crainic were able to make sense of the adjoint and coadjoint representations up to homotopy of a Lie algebroid as follows.

\begin{example}[The adjoint representation up to homotopy]\label{adjoint}
Let $\Alie$ be a Lie algebroid and $\nabla:\fX(M)\times\Gamma A\to \Gamma A$ any connection on $A\to M$. Define the adjoint representation up to homotopy by:
\begin{enumerate}
	\item The vector bundles $E_0=A,\ E_1=TM$, and $\partial=\rho$.
	\item The connections:
	\begin{equation*}
		\begin{array}{ll}
			\nabla^0_a b=[a,b]+\nabla_{\rho(b)}a,& a,b\in \Gamma A,\\
			\nabla^1_a X=[\rho(a),X]+\rho(\nabla_X a,)& a\in\Gamma A, \ X\in\fX(M).		
		\end{array}			
	\end{equation*}	 
	\item The two form:
	$$K(a,b)(X)=\nabla_X[a,b]-[\nabla_X a, b]-[a,\nabla_X b]+\nabla_{\nabla^1_a X}b-\nabla_{\nabla^1_b X}a, \qquad a,b\in\Gamma A, \ X\in\fX(M).$$
\end{enumerate}  
We denote by $D^{ad}$ the degree $1$ differential operator associated to the adjoint representation up to homotopy. 
\end{example}

\begin{example}[The coadjoint representation up to homotopy]\label{coadjoint}
	We define it as the dual representation of the adjoint:
\begin{enumerate}
	\item The vector bundles $E_{0}=T^*M,\ E_1=A^*$, and $\partial=\rho^*$.
	\item The connections:
	\begin{equation*}
		\begin{array}{ll}
			\langle\nabla^{1}_a \beta, b\rangle=\rho(a)\langle\beta,b\rangle-\langle \beta,\nabla_a^0 b\rangle,& a,b\in \Gamma A, \ \beta\in\Gamma A^*,\\
			\langle\nabla^0_a \tau, X\rangle=\rho(a)\langle\tau,X\rangle-\langle \tau,\nabla_a^1 X\rangle,& a\in \Gamma A,\ X\in\fX(M), \ \tau\in\Omega^1(M).	
		\end{array}			
	\end{equation*}	
	where the connections on the right are the ones coming from the adjoint representation.
	\item The two form $K^*=K$, Since $\Hom(TM, A)=\Hom(A^*,T^*M)$. 
\end{enumerate}  
We denote by $D^{ad^*}$ the differential operator associated with the coadjoint representation. 
\end{example}
Observe that these definitions of the adjoint and coadjoint representations up to homotopy depend on the choice of an auxiliary connection on $A\to M$. Therefore, they are not unique but it was proved in \cite{cam:rep} that they are all quasi-isomorphic.

\subsection{VB-algebroids and Q-bundles}
Once we know what a $2$-term representation up to homotopy is, we introduce the first approach to semi-direct products.
 
A \emph{VB-algebroid} is a double vector bundle, $D\to E$ over $A\to M$, equipped with a Lie algebroid structure on $D\to E$ compatible with the vector bundle structure of $D\to A$, see \cite{gra:vb} for a complete definition and compare with the original one of an LA-vector bundle \cite{mac:dou}. In this case, $A\to M$ inherits a natural Lie algebroid structure.

On a double vector bundle, $D\to E$ over $A\to M$, we denote the space of sections of $D$ over $E$ by $\Gamma(D,E)$. There are two special types of sections, called \emph{linear}, $\Gamma_l(D,E)=\{S\in\Gamma(D,E) \ | S \text{ is a vector bundle map }\}$, and \emph{core} $\Gamma_c(D,E)=\{S\in\Gamma(D,E) \ | S \text{ covers the zero section and constant on the fibres } \}$. The space of sections $\Gamma(D,E)$ is generated by linear and core sections as a $C^\infty(M)$-module.
Given a double vector bundle define the \emph{core} bundle, $C\to M$, as the intersection of the kernels of the two projections and the \emph{fat} bundle, $\widehat{A}\to M$ is a vector bundle that fits in the exact sequence
\begin{equation}\label{fat}
\Hom(E,C)\to \widehat{A}\to A
\end{equation}
and with $\Gamma \widehat{A}$ being isomorphic to the linear sections, $\Gamma_l(D, E)$. It also holds that $\Gamma C$ is isomorphic to the space of core sections, $\Gamma_c(D,E)$.

Any $2$-term representation up to homotopy of $(A\to M,[\cdot,\cdot],\rho)$ defines a $VB$-algebroid on $A\oplus C\oplus E\to E$ over $A\to M$. A Theorem of \cite{gra:vb} says that, given a $VB$-algebroid any splitting of \eqref{fat} defines a $2$-term representation up to homotopy of $A\to M$ on $E_0=C, \ E_1=E$. This explains the arrow ``$S_1$'' in \eqref{diagram}.

\begin{example}\label{VBCo-Tan}
Given any Lie algebroid $(A\to M,[\cdot,\cdot],\rho)$, the tangent and cotangent prolongations, $TA$ and $T^*A$, are examples of VB-algebroids. In both cases the fat bundle is the jet prolongation bundle of $A$, $\widehat{A}=J^1A$, and a splitting of \eqref{fat} is the same as a connection on $A\to M$. The representations up to homotopy obtained from the tangent and cotangent prolongations are the adjoint and coadjoint representations as defined in Examples \ref{adjoint} and \ref{coadjoint}, see also \cite{gra:vb}.
\end{example}

The next step is to describe VB-algebroids in terms of graded manifolds. For that, we need the following definition.

Let $p:\cE\to \cM$ be a vector bundle in the category of graded manifolds, see \cite{kot:char, raj:tes}. We say that it is a \emph{$Q$-bundle} if $(\cE, Q)$ is a $Q$-manifold and the vector field $Q$ is $p$-projectable, i.e. there exists $\widehat{Q}\in\fX^1(\cM)$ such that the following diagram commutes:
	\begin{equation*}
		\begin{array}{ccc}
		T[1]\cE&\xrightarrow{dp} &T[1]\cM\\
		Q\uparrow&&\uparrow\widehat{Q}\\
		\cE&\xrightarrow{p}&\cM
		\end{array}
	\end{equation*}

As proved in \cite{gra:vb} a double vector bundle $D$ is a VB-algebroid if and only if $p:D[1]_E\to A[1]$ is a $Q$-bundle and $Q$ is linear on the fibres of $p:D[1]_E\to A[1]$. Here the subscript on the shifting indicates the base of the vector bundle for which we are shifting the fibre coordinates. This gives to ``$F_1$'' in \eqref{diagram}.

\begin{remark}
The idea that $Q$-bundles over $A[1]$ with linear vector fields were representations of the Lie algebroid $(A[1], \widehat{Q})$ appears in \cite{vai:lie} and was developed in \cite{raj:mod}.
\end{remark}

Let us define the arrow ``$F_2$'' of \eqref{diagram}. Suppose that $p:\cE\to \cM$ is a vector bundle in the category of graded manifolds. We can define a new graded manifold $\cE[k]_\cM$, that also defines a vector bundle over $\cM$, where the fibre coordinates are shifted by $k$. 

If $p:\cE\to \cM$ is a $Q$-bundle with $Q$ linear on the fibres, then $p:\cE[k]_\cM\to \cM$ is also a $Q$-bundle. This follows from the fact that since $Q\in\fX^1(\cE)$ is $p$-projectable and linear on the fibres, locally it can be written as
\begin{equation*}
 Q=f^i(x) \dfrac{\partial}{\partial x^i}+ g_j^l(x) e^j \dfrac{\partial}{\partial e_l}\quad \text{ with }|f^i|-|x^i|=1 \quad |g_j^l|+|e^j|-|e^l|=1,
\end{equation*}
where $f^i(x),g_j^l(x)\in C_\cM, \ \{x^i\}$ are local coordinates of $\cM$ and $\{ e^j\}$ fibre coordinates of $\cE$. So once we shift the degree of the  coordinates $e^j$ by $k$, the same expression defines a vector field on $\cE[k]_\cM$ that is $p$-projectable, linear on the fibres, of degree $1$ and squares to zero.

\begin{corollary}\label{VBtoQ-bundle}
Let $D\to E$ be a VB-algebroid over $A\to M$. For any $k\in\bN^*$, the graded manifolds  $p:D[1]_E[k]_{A[1]}\to A[1]$ define $Q$-bundles.
\end{corollary}

\begin{examples}
Given any Lie algebroid $(A\to M, \rho, [\cdot,\cdot])$ the tangent and cotangent prolongations have natural structures of VB-algebroids. So we can consider the manifolds $TA[1]_{TM}[k]_{A[1]}=T[k]A[1]$ and $T^*A[1]_{A^*}[k]_{A[1]}=T^*[k+1]A[1]$. By the previous result we obtain that
\begin{equation*}
p: (T[k]A[1],\cL_Q)\to (A[1],Q) \ \text{ and }\ q: (T^*[k+1]A[1],\cL_Q=X_\theta)\to (A[1],Q)
\end{equation*}
are $Q$-bundles. As we already mentioned in Example \ref{VBCo-Tan}, these $Q$-manifolds are related to the adjoint and coadjoint representation up to homotopy of the Lie algebroid $\Alie$.
\end{examples}
Before finishing this section, we give a definition of the graded manifolds $D[1]_E[k]_{A[1]}$ in terms of vector bundles. In the case of the cotangent prolongation it coincides with the one given for $T^*[k+1]A[1]$ in Proposition \ref{VBchar}.

\begin{proposition}\label{VBdesc_D[1][k]}
Given a double vector bundle $D\to E$ over $A\to M$, define for any $k\in\bN^*$ the degree $k+1$ manifold $D[1]_E[k]_{A[1]}=(M, C_{D[1]_E[k]_{A[1]}})$. Then

\begin{itemize}
	\item For $k=1$,  $C_{D[1]_E[1]_{A[1]}}$ is generated by:
	\begin{equation*}
	C^i_{D[1]_E[1]_{A[1]}}=\left\{\begin{array}{ll}
	C^\infty(M),& i=0, \\
	\Gamma(A^*\oplus E^*),& i=1,\\
	\Gamma(\wedge^2 A^*\oplus \widehat{C}^*\oplus \wedge^2 E^*),& i=2.
	\end{array}
	\right.
	\end{equation*}
			\item For $k>1$,  $C_{D[1]_E[k]_{A[1]}}$ is generated by:
	\begin{equation*}
	C^i_{D[1]_E[k]_{A[1]}}=\left\{\begin{array}{ll}
	C^\infty(M),& i=0, \\
	\Gamma(\wedge^i A^*),& i<k,\\
	\Gamma(E^*\oplus \wedge^k A^*),& i=k,\\
	\Gamma(\widehat{C}^*\oplus\wedge^{k+1}A^*),& i=k+1.
	\end{array}
	\right.
	\end{equation*}
\end{itemize} 
where $\Gamma\widehat{C}^*=\Gamma_l(D^{*_A}, C^*)$ and $D^{*_A}\to C^*$ denotes the dual double vector bundle over $A\to M$.
\end{proposition}

\subsection{$L_\infty$-algebroids and $Q$-manifolds}
The second approach to the semi-direct product of a Lie algebroid with a representation up to homotopy uses the concept of $L_\infty$-algebroids. 

	An \emph{$L_k-$algebroid} $(\mathbf{A}\to M, l_i, \rho)$ is a non positively graded vector bundle $\mathbf{A}=\bigoplus\limits_{i=0}^{k-1}A_{-i}$ together with a bundle map $\rho:A_0\to TM$ and graded antisymmetric brackets $l_i:\Gamma(\mathbf{A}\times\cdots^i\times \mathbf{A})\to \Gamma \mathbf{A}$, with $i\in\{1,\cdots, k+1\}$, satisfying linearity and Jacoby-like identities; see \cite{bon:on} for a complete definition.
	
The definition involves shuffles for the higher Jacobi identities and is not easy to deal with. But in the examples that we are interested in almost all brackets are zero and we can compute all the higher identities explicitly. 

Let $\Alie$ be a Lie algebroid and $(E_0\to E_1,\partial, \nabla^0, \nabla^1, K)$ a representation up to homotopy. In \cite{chen:hig}, it is shown that $\mathbf{A}=(A_0=A\oplus E_1)\oplus(A_{-1}=E_0)$ 
inherits a $L_2$-algebroid structure with brackets given by
\begin{equation*}
\left\{\begin{array}{l}
\rho=\rho\circ p_A, \quad l_1=\partial, \\
l_2(a+e, b+e')=\left[a,b\right]+\nabla_a^1 e'-\nabla_b^1e, \quad l_2(a+e, \xi)=\nabla^0_a\xi,\\
l_3(a+e, b+e', c+e'')=-K(a,b)e''+K(a,c)e'-K(b,c)e.
\end{array}\right.
\end{equation*} 
where $ a,b,c\in\Gamma A, \ e, e', e''\in\Gamma E_1, \ \xi\in\Gamma E_0$ and $p_A:A\oplus E_1\to A$. We denote this $L_2$-algebroid by $A\ltimes(E_0\to E_1)[1]$. Moreover, $A\ltimes (E_0\to E_1)[1]$ is a Lie extension of $A\to M$ in the sense of \cite{chen:hig}. Formally we are shifting the representation up to homotopy by $1$ and after that taking the semi-direct product, which explains our notation. 

We now discuss $L_k$-algebroids concentrated in $3$ degrees.

\begin{example}
Fix $k> 3$ and consider $\mathbf{A}=\oplus_i A_{i}$ a graded vector bundle where $A_i=0$ if $i\not\in\{0,-k+1,-k+2\}$. An $L_k-$algebroid structure on $\mathbf{A}$ is the same as:
	\begin{itemize}
		\item The anchor $\rho:A_0\to TM$.
		\item The $l_1$ given by a bundle map $\partial: A_{-k+1}\to A_{-k+2}$.
		\item The $l_2$ given by:
			\begin{itemize}
				\item A bracket $[\cdot,\cdot]:\Gamma A_0\times \Gamma A_0\to \Gamma A_0$.
				\item $A_0$-connections $\Phi:\Gamma A_0\times \Gamma A_{-k+2}\to \Gamma A_{-k+2}$ and $\Psi:\Gamma A_0\times \Gamma A_{-k+1}\to \Gamma A_{-k+1}$. 
			\end{itemize}		
		\item The $l_3$ given by $[\cdot,\cdot,\cdot]_3:\Gamma A_0\times \Gamma A_0\times \Gamma A_{-k+2}\to \Gamma A_{-k+1}$ that is $C^\infty(M)$-linear.
		\item The $l_k$ given by $[\cdot,\cdots,\cdot]_k:\Gamma A_0\times \cdots\times\Gamma A_0\to \Gamma A_{-k+2}$ that is $C^\infty(M)$-linear.
		\item The $l_{k+1}$ given by $[\cdot,\cdots,\cdot]_{k+1}:\Gamma A_0\times \cdots\times\Gamma A_0\to \Gamma A_{-k+1}$ that is $C^\infty(M)$-linear.	 
	\end{itemize}
	satisfying the following Jacobi-like identities:
	\begin{equation}\label{eqLnalg}
	\left\{
		\begin{array}{l}
			0=\left[a_1,\left[a_2,a_3\right]_2\right]_2+c.p.\\
			\Phi_a(\partial(X))=\partial(\Psi_a X).\\
			\left[a_1, a_2,\partial(X)\right]_3= F^\Psi(a_1, a_2) X.\\
			\partial([a_1,a_2,\xi])=F^\Phi(a_1, a_2) \xi.\\
			0=([[a_1,a_2]_2,a_3,\xi]_3+c.p.)+([a_1,a_2,\Phi_{a_3}\xi]_3+c.p.)+(\Psi_{a_1}[a_2,a_3,\xi]_3+c.p.).\\
			0=\partial([a_1,\cdots,a_{k+1}]_{k+1}+c.p.)+([[a_1,a_2]_2,a_3,\cdots,a_{k+1}]_k+c.p.)+\\
			\qquad+(\Phi_{a_1}[a_2,\cdots,a_{k+1}]_k+c.p.).\\
			0=(\Psi_{a_1}[a_2,\cdots,a_{k+2}]_{k+1}+c.p.)+([[a_1,a_2]_2,a_3,\cdots,a_{k+2}]_{k+1}+c.p.)+\\
			\qquad +([a_1,a_2,[a_3,\cdots,a_{k+2}]_k]_3 +c.p.).
		\end{array}
	\right.	
	\end{equation}
		where $a_1,\cdots,a_{k+2}\in\Gamma A_0,\ \xi\in\Gamma A_{-k+2}, \ X\in\Gamma A_{-k+1}$.
\end{example}

The case $k=3$ is more complicated but can be trated similarly. 

\begin{proposition}\label{L_kalg}
Let $\Alie$ be a Lie algebroid and $(E_0\to E_1,\partial, \nabla^0,\nabla^1, K)$ a representation up to homotopy. Then, for any $k>2$ the graded vector bundle
\begin{equation*}
\mathbf{A}=\left\{\begin{array}{l}
A_0=A,\\
A_{-k+1}=E_0,\\
A_{-k+2}=E_1.
\end{array}\right.
\end{equation*}
inherits an $L_k$-algebroid structure with brackets given by
\begin{equation*}
\left\{\begin{array}{l}
\rho=\rho, \quad \partial=\partial,\\
\left[\cdot,\cdot\right]_2=\left[\cdot,\cdot\right], \quad \Phi=\nabla^1, \quad\Psi=\nabla^0,\\
\left[a,b,e\right]_3=K(a,b)e, \quad \left[\cdot,-,\cdot\right]_{k}=0, \quad \left[\cdot,-,\cdot\right]_{k+1}=0.\\
\end{array}\right.
\end{equation*}
We denote this $L_k$-algebroid by $A\ltimes(E_0\to E_1)[k-1]$. Moreover, $A\ltimes(E_0\to E_1)[k-1]$ are Lie extensions of $A\to M$.
\end{proposition}

\begin{proof}
Observe that all the objects have the right skew-symmetry and $C^\infty(M)$ linearity properties. Let us check the Jacobi-like identities \eqref{eqLnalg}. The first one is satisfied because it is the Jacobi identity for the Lie bracket on $A\to M$. The following four are equivalent to $(E_0\to E_1,\partial, \nabla^0,\nabla^1, K)$ being a representation up to homotopy of $A\to M$. Finally, since the $k$ and the $k+1$ brackets are zero the last two equations are trivially satisfied.

The last assertion is straightforward.
\end{proof}

Up to now, given a Lie algebroid and a representation up to homotopy we constructed, for any $k\in\bN^*$, an $L_{k+1}$-algebroid that we denoted by $A\ltimes(E_0\to E_1)[k]$ and is an extension of the Lie algebroid $\Alie$ in the sense of \cite{chen:hig}. This corresponds to the arrow ``$G_1$'' in \eqref{diagram}. 

In the case of the coadjoint representation up to homotopy we can incorporate the $H$-twist appearing in Section \ref{S4}:

\begin{corollary}
Let $\Alie$ be a Lie algebroid and $H\in\Gamma\bigwedge^{k+1}A^*$. Given a connection on $A\to M$ consider the coadjoint representation up to homotopy as defined in Example \ref{coadjoint}. For any $k>2$ the graded vector bundle
\begin{equation*}
\mathbf{A}=\left\{\begin{array}{l}
A_0=A,\\
A_{-k+1}=T^*M,\\
A_{-k+2}=A^*.
\end{array}\right.
\end{equation*}
inherits an $L_k$-algebroid structure with brackets given as before for $l_1, l_2$ and $l_3$ and
\begin{equation*}
\left\{\begin{array}{l}
\left[a_1,\cdots,a_k\right]_k=i_{a_k}\cdots i_{a_1}H,\\
\left[a_1,\cdots, a_{k+1}\right]_{k+1}=d(i_{a_{k+1}}\cdots i_{a_1}H)-\sum_{i=1}^{k+1}(-1)^{i+1}\langle D(a_i), i_{a_{k+1}}\cdots \widehat{i_{a_i}}\cdots i_{a_1}H\rangle.
\end{array}\right.
\end{equation*}
We denote these $L_k$-algebroids by $A\ltimes_H(A^*\to T^*M)[k-1]$.
\end{corollary}

\begin{proof}
By Proposition \ref{L_kalg} it remains to prove that the $k$ and the $k+1$ brackets satisfy the Jacobi like identities. Given $H\in\Gamma \bigwedge^{k+1}A^*$ consider $H\otimes 1$ in the coadjoint complex. The last two equations on \eqref{eqLnalg} are equivalent to $D^{ad^*}(H\otimes 1)=0$ and if we remember the compatibility between $D^{ad^*}$ and $d_A$ we obtain
\begin{equation*}
D^{ad^*}(H\otimes 1)=d_AH \otimes 1 +(-1)^{k+1}H\otimes D^{ad^*}(1)=d_AH,
\end{equation*}
so the $k$ and $k+1$ brackets satisfy the Jacobi like identities iff $d_AH=0$. 
\end{proof}

The arrow ``$G_2$'' in \eqref{diagram} was constructed in \cite{bon:on}, and gives an equivalence of categories between $L_k$-algebroids and $Q$-manifolds. Therefore given a Lie algebroid $\Alie$ and a representation up to homotopy $(E_0\to E_1, \partial, \nabla^0,\nabla^1, K)$, for any $k\in \bN^*$,  $\big(A\ltimes(E_0\to E_1)[k]\big)[1]$ is a $Q$-manifold. In fact, since $A\ltimes(E_0\to E_1)[k]$ is a Lie extension of $A\to M$, then $\big(A\ltimes(E_0\to E_1)[k]\big)[1]$ is a $Q$-bundle over $(A[1], Q)$. The only thing that remains to be proven is that these graded $Q$-bundles are the split version of the graded $Q$-bundles obtained in Corollary \ref{VBtoQ-bundle}.

Let $D\to E$ over $A\to M$ be a VB-algebroid with core $C\to M$ and consider $S:A\to \widehat{A}$ a splitting of the exact sequence \eqref{fat}. As we already mentioned, $A\to M$ is a Lie algebroid and $(C\to E, \partial, \nabla^0, \nabla^1, K)$ is a representation up to homotopy of $A\to M$, that depends on the splitting $S:A\to\widehat{A}$.

\begin{theorem}
 For any $k\in\bN^*$, the splitting $S:A\to \widehat{A}$ induces a $Q$-manifold isomorphism
\begin{equation*}
 (D[1]_E[k]_{A[1]},Q) \cong\big((A\ltimes(C\to E)[k])[1], Q\big).
\end{equation*}
\end{theorem} 

\begin{proof}

First recall that, by the theory of double vector bundles, given a splitting of the exact sequence \eqref{fat} we obtain a decomposition of the double vector bundle $D\to E$ over $A\to M$. A decomposition of the double vector bundle produces a decomposition of the dual double vector bundle $D^{*_A}\to C^*$ over $A\to M$. Finally, this decomposition induce a splitting of the sequence 
\begin{equation}\label{fatdual}
\Hom(A, E)\to \widehat{C}^*\to C^*.
\end{equation}
 Therefore, given $S:A\to\widehat{A}$ splitting of \eqref{fat} we obtain $S':C^*\to \widehat{C}^*$ splitting of \eqref{fatdual}.

Using Proposition \ref{VBdesc_D[1][k]} we see that the map $S':C^*\to\widehat{C}^*$ induces the following isomorphism between the graded manifolds $D[1]_E[k]_{A[1]}=(M,C_{D[1]_E[k]_{A[1]}})$ and $(A\ltimes(C\to E)[k])[1]=(M,C_{(A\ltimes(C\to E)[k])[1]})$:
\begin{equation*}
\left\{\begin{array}{ll}
\id:M\to M,\\
\Psi^i:C^i_{(A\ltimes(C\to E)[k])[1]}\to C^i_{D[1]_E[k]_{A[1]}},& \Psi^i=\id \text{ for } i<k.\\
\Psi^k:C^k_{(A\ltimes(C\to E)[k])[1]}\to C^k_{D[1]_E[k]_{A[1]}},& \Psi^k(c+\alpha\cdot\xi+\tau)=S'(c)+m(\alpha,\xi)+\tau.
\end{array}\right.
\end{equation*}
where $c\in\Gamma C^*, \ \alpha\in\Gamma A^*, \xi\in\Gamma E$ and $\tau\in\Gamma\bigwedge^k A^*$. 

One can see that this rules defines a graded isomorphism and going along the proof of the main Theorem in \cite{bon:on} it is easy to conclude that this map also commutes the $Q$-structures.
\end{proof}

\begin{corollary}
For any $k\in\bN^*$, any connection on $A\to M$ induces a $Q$-manifold isomorphism
\begin{equation*}
(T^*[k+1]A[1],X_{\theta_H})\cong\big((A\ltimes_H(A^*\to T^*M)[k])[1],Q\big).
\end{equation*}
\end{corollary}
 
\section{AKSZ $\sigma$-models and integrating objects}\label{S6}

\subsection{AKSZ $\sigma$-models}

In physics, a gauge theory is a field theory with a lagrangian action functional which is invariant under some symmetries. If the symmetries are not given by a Lie algebra, Batalin and Vilkovisky propose a way to enlarge the space of fields using supermanifolds that allows to compute some physically interesting quantities of the system.

In the celebrated paper \cite{AKSZ} Alexandrov, Kontsevich, Schwarz and Zaboronsky show how to produce a BV theory with space of superfields the infinite dimensional $\bZ$-graded manifold  $Maps(T[1]\Sigma, \cM)$, where $\cM$ is a symplectic $Q$-manifold of degree $k$ and $\Sigma$ a $k+1$ dimensional smooth manifold, see \cite{mnev:lec} for more details. The gauge theories that are constructed in this way are known as AKSZ $\sigma$-models and examples of them include the $A$ and $B$ models, the Poisson $\sigma$-model, Chern-Simons theory among others.

As a consequence, we have that our manifolds $(T^*[k]A[1],\{\cdot,\cdot\},X_{\theta_H})$ can be considered as targets of AKSZ  $\sigma$-models with source $k+1$ dimensional manifolds. For some cases these gauge theories were already considered: when $A\to M$ is just a Lie algebra, it is called BF theory and has been extensively studied in \cite{cat:top, cat:loop, cat:hig}. When $A=TM$ and $k=2$,  it corresponds to the standard Courant algebroid and this particular example was called the open topological membrane and studied in \cite{hof:top}. The general $k=2$ fits in the Courant $\sigma$-model as described in \cite{roy:AKSZ}, but not many explicit computations have been done.

Following \cite{cat:cla}, we obtain that for any manifold $\Sigma$, possibly with boundary, of dimension $k+1$ the space $Maps(T[1]\Sigma, T^*[k]A[1])$ defines a BV-BFV theory using the AKSZ construction, see \cite[section 6]{cat:cla}. Now, what we impose on the boundary is that the image of any boundary component takes values on a lagrangian $Q$-submanifold, or what is equivalent by Theorem \ref{Nan-Di-Qlag} on a higher Dirac structure. For more details about this see \cite[Section 3.7]{cat:cla}.

In forthcoming works, we plan to develop the full theory for this kind of $\sigma$-models and compute perturbative topological invariants, associated to Lie algebroids, for manifolds of any dimension using the techniques of \cite{cat:per}. 
\begin{remark}
Comparing to other cases, note that the Topological open $k$-brane of \cite{hof:top} has a different target manifold. For example when $k=3$, the target is the manifold $T^*[3]T^*[2]T[1]M$, which is bigger than $T^*[3]T[1]M$. Also the Nambu $\sigma$-model proposed in \cite{bou:aks} has a different target manifold, in this case $T^*[k]\big((\bigwedge^{k-1}T)[k-1]T[1]M\big)$. In the AKSZ $\sigma$-model proposed here the Nambu geometry appears naturally as boundary contributions not in the action, as in \cite{bou:aks}.
\end{remark}

\subsection{Speculations on the integration of Q-manifolds} 
The works of \v{S}evera \cite{sev:der} and \cite{sev:int} propose that degree $k$ $Q$-manifolds integrate to Lie $k$-groupoids. Lie $k$-groupoids are understood here as simplicial manifolds that satisfy the Kan condition for any $l>k$.

The manifolds $(T^*[k]A[1],\{\cdot,\cdot\},\theta_H)$ are examples of $Q$-manifolds and we can apply \v{S}evera's procedure to integrate them. In particular, for $k=2$ we are discussing the integration of the double of a Lie algebroid seen as a Courant algebroid; this integration is not completely understood yet but some particular cases have been treated in the literature,  see \cite{li:int, meh:sym, chen:hig}. The integrations of lagrangians $Q$-submanifolds of $(T^*[2]T[1]M,\{\cdot,\cdot\},\theta_H)$ include the integrations of Dirac structures to lagrangian subgroupoids as shown in \cite{meh:sym}, that are expected to correspond to presymplectic groupoids as in \cite{bur:int}. For $k>2$ what happens should be similar to this case.

The interesting consequence is that now we have a place where Nambu structures must be integrated. As Poisson structures integrate to symplectic groupoids that corresponds to lagrangian subgroupoids inside symplectic $2$-groupoids, Nambu structures should integrate to lagrangian $k-1$-subgroupoid of symplectic $k$-groupoids. Since Nambu structures just define weak lagrangians the lagrangian $k-1$-subgroupoid could be singular. 

\begin{remark}
Another approach to integration in our particular case could be using the fact that we have a semidirect product. We could try to integrate first the algebroid and after that a cocycle given by the  representation; these techniques were used in \cite{chen:hig} but then it is unknown how to obtain the symplectic structure.
\end{remark}
\begin{remark}
Recently it was proposed by Laurent-Gengoux and Wagemann that Leibniz algebroids will be integrated to Lie rackoids \cite{lau:rac}. Nambu-Dirac structures are examples of Leibniz algebroids, see \cite{hag:nam}. So it would be interesting to compare, if possible, this rackoid integration with the one proposed here.
\end{remark}


\begin{thebibliography}{10}

\bibitem{cam:rep}
C.-A.~Abad, and M.~Crainic, \emph{Representations up to homotopy of Lie algebroids}.  J. Reine Angew. Math. \textbf{663} (2012), 91--126.

\bibitem{AKSZ}
M.~Alexandrov, A.~Schwarz, M.~Kontsevich, and O.~Zaboronsky, \emph{The geometry of the master equation and topological quantum field theory}. Internat. J. Modern Phys. A \textbf{12} (1997), no. 7, 1405--1429.

\bibitem{bi:on}
Y.~Bi, and Y.~Sheng, \emph{On higher analogues of Courant algebroids}. Sci. China Math., \textbf{54} (3) (2011), 437--447.

\bibitem{bi:dir}
\bysame, \emph{Dirac structures for higher analogues of Courant algebroids}. Int. J. Geom. Methods Mod. Phys. \textbf{12} (2015), no. 1, 1550010, 13 pp.

\bibitem{bon:on}
G.~Bonavolonta, and N.~Poncin, \emph{On the category of Lie n-algebroids}. J. Geom. Phys. \textbf{73} (2013), 70--90.

\bibitem{zab:bra}
G.~Bonelli, and M.~Zabzine, \emph{From current algebras for p-branes to topological M-theory}.J. HighEnergy Phys., (9):015, 20, 2005.

\bibitem{bou:aks}
P.~Bouwknegt, and B.~Jur\v{c}o \emph{AKSZ construction of topological open p-brane action and Nambu brackets}.  Rev. Math. Phys. \textbf{25} (2013), no. 3, 1330004, 31 pp.

\bibitem{bur:super}
H.~Bursztyn, A.-S.~Cattaneo, R.A.~Mehta, and M.~Zambon, \emph{Graded geometry and Generalized reduction}, work in progress.

\bibitem{bur:int}
 H.~Bursztyn, M.~Crainic, A.~Weinstein, and C.~Zhu, \emph{Integration of twisted Dirac brackets}. Duke Math. J. \textbf{123} (2004), no. 3, 549--607.

\bibitem{bur:frob}
H.~Bursztyn, M.~Cueca, and R.A.~Mehta, \emph{Frobenius theorem for graded manifolds}. work in progress.

\bibitem{bur:hig}
H.~Bursztyn, N.~Martinez, and R.~Rubio, \emph{On higher Dirac Structures}.  Int. Math. Res. Not. IMRN (2019), no. 5, 1503--1542.

\bibitem{bur:sup}
H.~Bursztyn, D.~Iglesias-Ponte, and P.~\v{S}evera, \emph{Courant morphisms and moment maps}. Math. Res. Lett. \textbf{16} (2009), no. 2, 215--232. 

\bibitem{can:on}
F.~Cantrijn, A.~Ibort, M.~de Le\'{o}n, \emph{On the geometry of multisymplectic manifolds}. J. Austral. Math. Soc. Ser. A \textbf{66} (1999), no. 3, 303--330.

\bibitem{cat:top}
A.-S.~Cattaneo, P.~Cotta-Ramusino, J.~Fr\"{o}hlich, and M.~Martellini, \emph{Topological BF theories in 3 and 4 dimensions}. J. Math. Phys. \textbf{36} (1995), no. 11, 6137--6160.

\bibitem{cat:loop}
 A.-S.~Cattaneo, P.~Cotta-Ramusino, and C.-A.~Rossi, \emph{Loop observables for BF theories in any dimension and the cohomology of knots}. Lett. Math. Phys. \textbf{51} (2000), no. 4, 301--316.

\bibitem{cat:hig}
A.-S.~Cattaneo, and C.-A.~Rossi, \emph{Higher-dimensional BF theories in the Batalin-Vilkovisky formalism: the BV action and generalized Wilson loops}. Comm. Math. Phys. \textbf{221} (2001), no. 3, 591--657.

\bibitem{cat:cla}
A.-S.~Cattaneo, P.~Mnev, and N.~Reshetikhin, \emph{Classical BV theories on manifolds with boundary}. Comm. Math. Phys. \textbf{332} (2014), no. 2, 535--603.

\bibitem{cat:per}
\bysame, \emph{Perturbative quantum gauge theories on manifolds with boundary}. Comm. Math. Phys. \textbf{357} (2018), no. 2, 631--730.

\bibitem{cat:int}
 A.-S.~Cattaneo, and F.~Sch\"{a}tz, \emph{ Introduction to supergeometry}. Rev. Math. Phys. \textbf{23} (2011), no. 6, 669--690.
 
\bibitem{cou:bey}
 T.~Courant, and A.~Weinstein, \emph{Beyond Poisson structures}. Action hamiltoniennes de groupes. Troisième théorème de Lie (Lyon, 1986), 39--49, Travaux en Cours, \textbf{27}, Hermann, Paris, 1988.

\bibitem{duf:poi}
J-P.~Duflo, and N.T.~Zung, \emph{Poisson structures and their normal forms}. Progress in Mathematics, \textbf{242}. Birkhäuser Verlag, Basel, (2005). xvi+321 pp.

\bibitem{gra:vb}
A.~Gracia-Saz and R.~Mehta, \emph{Lie algebroid structures on double vector bundles and representation theory of Lie algebroids}. Adv. Math. \textbf{223} (2010), no. 4, 1236--1275.
  
\bibitem{gru:h-t}
 M.~Gr\"{u}tzmann, \emph{H-twisted Lie algebroids}. J. Geom. Phys. \textbf{61} (2011), no. 2, 476--484.


\bibitem{gua:gen}
M.~Gualtieri, \emph{Generalized complex geometry}. PhD Thesis arXiv:math/0401221

\bibitem{hag:nam}
Y.~Hagiwara, \emph{Nambu-Dirac manifolds}.  J. Phys. A \textbf{35} (2002), no. 5, 1263--1281.

\bibitem{hit:gen}
N.~J.~Hitchin, \emph{Generalized Calabi-Yau manifolds}. Q. J. Math. \textbf{54} (2003), no. 3, 281--308.

\bibitem{hof:top}
C.~Hofman, and J-S.~Park, \emph{BV quantization of topological open membranes}. Comm. Math. Phys. \textbf{249} (2004), no. 2, 249--271.

\bibitem{ike:pq3}
N.~Ikeda, and K.~Uchino, \emph{QP-structures of degree 3 and 4D topological field theory}. Comm. Math. Phys. \textbf{303} (2011), no. 2, 317--330.

\bibitem{kos:qua}
Y.~Kosmann-Schwarzbach, \emph{Quasi, twisted, and all that... in Poisson geometry and Lie algebroid theory}. Progr. Math., \textbf{232} (2005), Birkhäuser Boston.

\bibitem{kot:char}
 A.~Kotov, and T.~Strobl, \emph{Characteristic classes associated to $Q$-bundles}. Int. J. Geom. Methods Mod. Phys. \textbf{12} (2015), no. 1, 1550006, 26 pp.

\bibitem{lau:rac}
C.~Laurent-Gengoux, and F.~Wagemann, \emph{Lie rackoids}. Ann. Global Anal. Geom. \textbf{50} (2016), no. 2, 187--207.

\bibitem{li:int}
 D.~Li-Bland, and P.~\v{S}evera, \emph{Integration of exact Courant algebroids}. Electron. Res. Announc. Math. Sci. \textbf{19} (2012), 58-–76.
 
\bibitem{mac:dou}
 K.-C.-H.~Mackenzie, \emph{Double Lie algebroids and second-order geometry. II}. Adv. Math. \textbf{154} (2000), no. 1, 46--75. 
 
\bibitem{raj:tes}
R.~Mehta, \emph{Supergroupoids, double structures, and equivariant cohomology}. Thesis (Ph.D.)–University of California, Berkeley. 2006. 133 pp.

\bibitem{raj:mod}
\bysame,  \emph{Lie algebroid modules and representations up to homotopy}. Indag. Math. (N.S.) \textbf{25} (2014), no. 5, 1122--1134.

\bibitem{meh:sym}
 R.~Mehta, and X.~Tang, \emph{Symplectic structures on the integration of exact Courant algebroids}. J. Geom. Phys. \textbf{127} (2018), 68--83.

\bibitem{mnev:lec}
P.~Mnev, \emph{Lectures on Batalin-Vilkovisky formalism and its applications in topological quantum field theory}. arXiv:1707.08096 [math-ph]

\bibitem{rog:inf}
C.~Rogers, \emph{$L_\infty-$algebras from multisymplectic geometry}. Lett. Math. Phys. \textbf{100} (2012), no. 1, 29--50.

\bibitem{roy:on}
D.~Roytenberg, \emph{On the structure of graded symplectic supermanifolds and
Courant algebroids}, Quantization, Poisson brackets and beyond (Manchester, 2001), 169–-185, Contemp. Math., \textbf{315}, Amer. Math. Soc., Providence, RI, (2002).

\bibitem{roy:AKSZ}
\bysame, \emph{AKSZ-BV formalism and Courant algebroid-induced topological field theories}. Lett. Math. Phys. \textbf{79} (2007), no. 2, 143--159.

\bibitem{sev:some}
P.~\v{S}evera, \emph{Some title containing the words "homotopy'' and "symplectic'', e.g. this one}. Travaux mathématiques. Fasc. XVI, 121--137, Trav. Math., \textbf{16}, Univ. Luxemb., Luxembourg, 2005.

\bibitem{sev:der}
\bysame, \emph{$L_\infty-$algebras as $1-$jets of simplicial manifolds (and a bit beyond)}. Preprint: arXiv:math/0612349.

\bibitem{sev:int}
P.~\v{S}evera, and M.~\v{S}ira\v{n}, \emph{Integration of differential graded manifolds}. Preprint: arXiv:1506.04898.

\bibitem{chen:hig}
Y.~Sheng, and C.~Zhu, \emph{Higher Extensions of Lie Algebroids}.  Commun. Contemp. Math. \textbf{19} (2017), no. 3, 1650034, 41 pp.

\bibitem{tak:nam}
L.~Takhtajan, \emph{On foundation of the generalized Nambu mechanics}. Comm. Math. Phys. \textbf{160} (1994), no. 2, 295--315.

\bibitem{vai:lie}
A.~Vaintrob, \emph{Lie algebroids and homological vector fields}. (Russian) Uspekhi Mat. Nauk \textbf{52} (1997), no. 2(314), 161--162.

\bibitem{wad:nam}
A.~Wade, \emph{Nambu-Dirac structures for Lie algebroids}.  Lett. Math. Phys. \textbf{61} (2002), no. 2, 85--99.

\bibitem{zam:inf}
M.~Zambon, \emph{$L_\infty-$algebras and higher analogues of Dirac structures and Courant algebroids}. J. Symplectic Geom. \textbf{10} (2012), no. 4, 563--599.

\end{thebibliography}
\end{document}